%% file: weaksoft_final.tex
\documentclass[a4paper,11pt]{amsart}
\usepackage{amssymb,amsmath,amsthm}
\usepackage{fullpage}
\usepackage[colorlinks=true,linkcolor=blue, citecolor=red]{hyperref}
\usepackage{eucal}
\usepackage{color}
\usepackage{xcolor}
\usepackage{stackrel, enumerate}
\usepackage{graphicx, bbm}

\definecolor{darkblue}{rgb}{0.2,0.2,0.6}
\newcommand\sfp{\mathsf{p}}

\input commands.tex

\usepackage[left=2.7cm, right=2.7cm, marginparwidth=2cm, footskip = 0.8cm ]{geometry}
\usepackage[normalem]{ulem}
\definecolor{DarkGreen}{rgb}{0,0.5,0.1}

\definecolor{DarkBlue}{rgb}{0,0.1,0.5}

\newcommand\soutD{\bgroup\markoverwith
	{\textcolor{DarkGreen}{\rule[.5ex]{2pt}{1pt}}}\ULon}
\newcommand\soutP{\bgroup\markoverwith
	{\textcolor{blue}{\rule[.5ex]{2pt}{1pt}}}\ULon}
\newcommand{\Hm}[1]{\leavevmode{\marginpar{\tiny%
			$\hbox to 0mm{\hspace*{-0.5mm}$\leftarrow$\hss}%
			\vcenter{\vrule depth 0.1mm height 0.1mm width \the\marginparwidth}%
			\hbox to
			0mm{\hss$\rightarrow$\hspace*{-0.5mm}}$\\
			\relax\raggedright #1}}}

\newcommand{\OO}{\mathcal{O}}

\newtheorem{theorem}{Theorem}[section]

\newtheorem{lemma}[theorem]{Lemma}
\newtheorem{proposition}[theorem]{Proposition}

\begin{document}

\title[Bound states of weakly deformed soft waveguides]{Bound states of weakly deformed soft waveguides}

\author[P. Exner]{Pavel Exner}
\address[P. Exner]{Department of Theoretical Physics, Nuclear Physics Institute, Czech Academy of Sciences, 25068 \v Re\v z near Prague, Czechia, and Doppler Institute for Mathematical Physics and Applied Mathematics, Czech Technical University, B\v rehov\'a 7, 11519 Prague, Czechia}
\email{exner@ujf.cas.cz}

\author[S. Kondej]{Sylwia Kondej}
\address[S. Kondej]{Institute of Physics, University of Zielona G\'ora, ul.\ Szafrana 4a, \\ 65246 Zielona G\'ora, Poland} \email{s.kondej@if.uz.zgora.pl}

\author[V. Lotoreichik]{Vladimir Lotoreichik}
\address[V. Lotoreichik]{Department of Theoretical Physics, Nuclear Physics Institute, 	Czech Academy of Sciences, 25068 \v Re\v z, Czechia}
\email{lotoreichik@ujf.cas.cz}



\maketitle

\begin{abstract}
In this paper we consider the two-dimensional Schr\"odinger operator with an attractive potential which is a multiple of the characteristic function of an unbounded strip-shaped region, whose thickness is varying and is determined by the function $\mathbb{R}\ni x \mapsto d+\varepsilon f(x)$, where $d > 0$ is a constant, $\varepsilon > 0$ is a small parameter, and $f$ is a compactly supported continuous function. We prove that if $\int_{\mathbb{R}} f \,\mathsf{d} x > 0$, then the respective Schr\"odinger operator has a unique simple eigenvalue below the threshold of the essential spectrum for all sufficiently small $\varepsilon >0$ and we obtain the asymptotic expansion of this eigenvalue in the regime $\varepsilon\rightarrow 0$. An asymptotic expansion of the respective eigenfunction as $\varepsilon\rightarrow 0$ is also obtained. In the case that $\int_{\mathbb{R}} f \,\mathsf{d} x < 0$ we prove that the discrete spectrum is empty for all sufficiently small $\varepsilon > 0$. In the critical case $\int_{\mathbb{R}} f \,\mathsf{d} x = 0$, we derive a sufficient condition for the existence of a unique bound state for all sufficiently small $\varepsilon > 0$.
\end{abstract}

\section{Introduction}\label{s:intro}
\setcounter{equation}{0}
Spectral analysis of differential operators describing quantum particles confined to tubular shape regions attracted a lot of attention in recent decades.
Interesting relations were found linking spectral properties of such systems to their geometry; we refer to \cite{EK15} and the bibliography therein.
Two most often investigated models were Dirichlet tubes and `leaky wires', that is, Schr\"odinger operators with $\delta$-interactions supported on curves.
Analysis of these models can be applied to the description of numerous systems investigated experimentally such as semiconductor nanowires, cold atom waveguides, and many others.

When regarded as models of real physical systems, both Dirichlet tubes and `leaky wires' include idealizations. In the first case it is the hard wall preventing tunnelling between different parts of the guide, in the second one it is the zero width of such a `wire'. This motivated recently interest to another class of models, for which the name \emph{soft quantum waveguides} was coined, using Schr\"odinger operators in which the mentioned singular potential is replaced by a regular potential `ditch'. Here again a non-trivial geometry may give rise to a non-trivial discrete spectrum. Using the Birman-Schwinger principle, a sufficient condition was derived for the existence of bound states in two-dimensional soft waveguides \cite{Ex20}. Another possible approach is direct use of the variational method; in this way the discrete spectrum existence was established in guides of a particular shape usually labelled as `bookcover' \cite{KKK21}. Furthermore, one can prove a sort of isoperimetric inequality in this setting \cite{EL20}, see also \cite{EKP20, WT14} for related results.

Many other questions remain open, \cf~\cite{Ex20}. In this paper we address one of them, namely the behaviour of the discrete spectrum in the case of weak geometric perturbations. We analyse a particular case when a flat-bottom guide is a weak local deformation of a straight potential channel. Using the Birman-Schwinger principle we prove that such a system has a unique bound state provided that the weak deformation enlarges the potential support area and we derive the first two terms of the eigenvalue asymptotic expansion in terms of the perturbation parameter, together with the corresponding expansion of the eigenfunction. On the contrary, we show that bound states are absent if the interaction support is shrunk under the weak perturbation. Finally, in the critical case when the guide `area' is preserved we prove a sufficient condition for the existence of a unique bound state
under weak deformation; it is present, roughly speaking, when the perturbation is sufficiently extended in the longitudinal direction.

Our results are  generalizations for soft waveguides  of the classical results
on Dirichlet tubes obtained in  \cite{BGRS97} by Bulla, Gesztesy, Simon,   and Regner and in~\cite{EV97} by the first author and Vugalter; as in those papers we restrict ourselves to the  situation where the deformation is one-sided. The methods employed in the present paper are significantly different. In contrast to
locally deformed Dirichlet tubes the underline Hilbert space for soft waveguides is not varying in the course of deformation. In this paper, we are not using the change of coordinates employed in~\cite{BGRS97, EV97}. The analysis of the non-critical case reduces to careful inspection of the integral operator involved into the Birman-Schwinger principle and its convenient decomposition. In the critical case, we apply the min-max principle on a suitably chosen trial function.

The employed decomposition of the integral operator in the Birman-Schwinger principle is inspired by the analysis of the attractive $\dl$-interaction supported on weakly deformed straight lines in two dimensions~\cite{EK15} and weakly deformed planes in three dimensions~\cite{EKL18}.
\subsection{Geometry of the waveguide and the Hamiltonian}\label{s:geo}
\setcounter{equation}{0}
Let $f\colon\dR\arr\dR$ be a continuous compactly supported function and let $\eps \in [0, \eps_0)$ with $\eps_0 >0$ being sufficiently small.
As indicated above, the object of our interest is the Schr\"odinger operator with a flat-bottom ditch-shaped potential, the support of which is a planar strip
\begin{equation*}\label{eq:strip}
	\Omg_\eps := \big\{(x_1,x_2)\in\dR^2\colon
    0 < x_2 < d+\eps f(x_1)\big\}
\end{equation*}
as sketched in Figure~\ref{Fig-1}. Clearly, for all sufficiently small $\eps > 0$ the planar strip $\Omg_\eps$ is well defined.
\begin{figure}
\includegraphics[width=.70\textwidth]{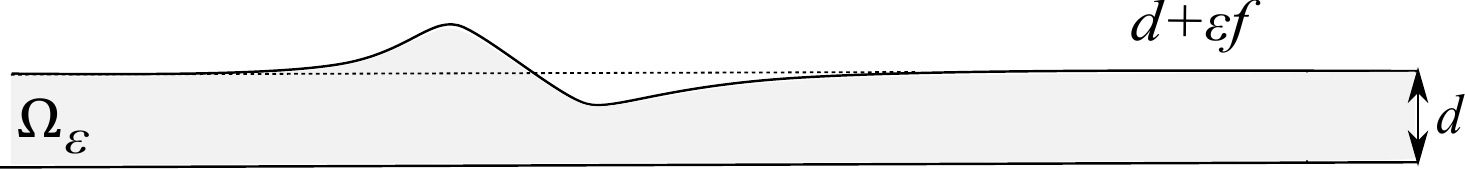}
\caption{Geometry of $\Omega_\varepsilon$.}
\label{Fig-1}
\end{figure}
Let the interaction strength $\aa > 0$ be fixed. We are interested in the spectral properties of the self-adjoint Schr\"odinger operator $\sfH_{\aa,\eps}$ acting in the Hilbert space $L^2(\dR^2)$ and defined by
\begin{equation}\label{eq:operator}
	\sfH_{\aa,\eps}u := -\Delta u -\aa\chi_{\Omg_\eps} u,\qquad\dom\sfH_{\aa,\eps} := H^2(\dR^2),
\end{equation}
where $\chi_{\Omg_\eps}\colon\dR^2\arr\dR$ is the characteristic function of the open set $\Omg_\eps$, that is, $\chi_{\Omg_\eps}(x) = 1$ for $x\in\Omg_\eps$ and $\chi_{\Omg_\eps}(x) = 0$ for $x\notin\Omg_\eps$.
The operator $\sfH_{\aa,\eps}$ represents
the closed, densely defined, symmetric, and lower-semibounded quadratic form defined in $L^2(\dR^2)$ by
\begin{equation}\label{eq:frhaaeps}
	\frh_{\aa,\eps}[u] := \int_{\dR^2}|\nb u|^2\dd x - \aa\int_{\Omg_\eps}|u|^2\dd x,\qquad \dom\frh_{\aa,\eps} := H^1(\dR^2).
\end{equation}

The free Hamiltonian $\sfH_{\aa,0}$ (referring to $\eps = 0$) can be analysed easily via separation of variables. In this case the underline planar strip $\Omg_0$ is straight. We introduce auxiliary simple one-dimensional Schr\"odinger operators
in the Hilbert space $L^2(\dR)$ as
\begin{equation}\label{eq-onedinHamil}
\begin{aligned}
\sfh_0 \psi & := -\psi'',\qquad
&\dom\sfh_0 &:= H^2(\dR),\\
\sfh_\aa \psi &:= -\psi'' -\aa\chi_{[0,d]}\psi,\qquad
&\dom \sfh_\aa &:= H^2(\dR),
\end{aligned}
\end{equation}
where $\chi_{[0,d]}\colon\dR\arr\dR$ is the characteristic function of the interval $[0,d]$. Clearly, we have $\s(\sfh_0) = \sess(\sfh_0) = [0,\infty)$. It follows from \cite[Lem. 9.35]{Te} that the essential spectrum of $\sfh_\aa$ coincides with the positive half-line, $\sigma_{\mathrm{ess}}(\sfh_\aa)=[0, \infty)$; what is important for us is the presence of the discrete spectrum. As a consequence of~\cite[Thm. 2.5]{Si76} the negative discrete spectrum of $\sfh_\aa$ is non-empty and by~\cite[Thm. 10.12.1 (8)-(iii)]{Z05} all the negative eigenvalues of $\sfh_\aa$ are simple. In view of \cite[Cor. 9.43]{Te} the dimension of the spectral subspace of $\sfh_\aa$ corresponding to the negative spectrum is finite and we denote by $\{\mu_n\}_{n=1}^N$, $N\in\dN$, the sequence of negative eigenvalues of $\sfh_\aa$, arranged in the ascending order, and by $\{v_n\}_{n=1}^N$ the corresponding real-valued eigenfunctions belonging to $H^2(\dR)$ and normalized in $L^2(\dR)$. It is easy to see that all the eigenfunctions $\{v_n\}_{n=1}^N$ are bounded.
For the sake of convenience we extend the sequence $\{\mu_n\}_{n=1}^N$ up to an infinite one by setting $0 =: \mu_{N+1} = \mu_{N+2} = \dots$.

The Hamiltonian $\sfH_{\aa,0}$ can be represented as the closure of $\sfh_\aa\otimes \sfI + \sfI\otimes \sfh_0$, where the tensor products
are understood with respect to the decomposition $L^2(\dR^2) = L^2(\dR)\otimes L^2(\dR)$ corresponding to the axes $x_1$ and $x_2$. Hence~\cite[Cor. 7.25]{S12} yields that
\begin{equation*}\label{key}
	\s(\sfH_{\aa,0})=\sess(\sfH_{\aa,0}) = [\mu_1,\infty).
\end{equation*}
We adopt the notation
\[
	\lm_1^\aa(\eps) := \inf\s(\sfH_{\aa,\eps})
\]
for the lowest spectral point of $\sfH_{\aa,\eps}$. We also denote by $N_\aa(\eps)$ the dimension of the spectral subspace of $\sfH_{\aa,\eps}$ corresponding to the interval $(-\infty,\mu_1)$.
The essential spectrum of $\sfH_{\aa,\eps}$ coincides with that of $\sfH_{\aa,0}$ and can be explicitly characterised as $\sess(\sfH_{\aa,\eps}) = [\mu_1,\infty)$ (see Proposition~\ref{ess}).

\subsection{Main results}\label{s:main}
As indicated above, the topic of this paper are the spectral properties of $\sfH_{\aa,\eps}$ in the limit $\eps \to 0$.
Our first result concerns the existence and asymptotic properties of the bound state induced by weak deformation in the non-critical regime $\int_\dR f(x_1)\,\dd x_1 \ne 0$.
\begin{theorem}\label{thm}
Let $\eps \in (0,\eps_0]$ and let the self-adjoint operator $\sfH_{\aa,\eps}$	in the Hilbert space $L^2(\dR^2)$ be defined as in~\eqref{eq:operator}. Then the following claims hold:
\begin{myenum}
\item If $\int_\dR f(x_1)\,\dd x_1 > 0$, then $N_\aa(\eps) = 1$ for all sufficiently small $\eps > 0$ and the lowest eigenvalue of $\sfH_{\aa,\eps}$
admits the asymptotic expansion
\begin{equation}\label{eq:asymptoticslm}
	\lm_1^\aa(\eps) = \mu_1 -\eps^2\left(\frac{\aa^2v_1^4(d)}{4}\right)\left(\int_\dR f(x_1)\,\dd x_1\right)^2 + \OO(\eps^3),\qquad \eps\arr0,
\end{equation}
where $v_1(d)$ is the value at the point $x=d$ of the real-valued normalized ground-state eigenfunction of $\sfh_\aa$ (defined in~\eqref{eq-onedinHamil}).
\item If $\int_{\dR} f(x_1) \,\dd x_1 < 0$, then $N_\aa(\eps) = 0$ for all sufficiently small $\eps > 0$, that is, for small $\eps > 0$ the discrete spectrum of $\sfH_{\aa,\eps}$ below $\mu_1$ is empty.
\end{myenum}
\end{theorem}
\begin{remark}\label{rem:limit}
	In order to compare the asymptotics in~\eqref{eq:asymptoticslm}  with the main result of~\cite{BGRS97} assume, in addition, that
	the function $f$ satisfying $\int_\dR f(x_1)\dd x_1 > 0$ belongs to $C^\infty_0(\dR)$. Recall that the Dirichlet Laplacian on $\Omg_\eps$ is defined via the first representation as the unique self-adjoint operator in $L^2(\Omg_\eps)$ associated with the quadratic form $H^1_0(\Omg_\eps)\ni u\mapsto \|\nb u\|^2_{L^2(\Omg_\eps;\dC^2)}$.
	It is proved in~\cite[Thm. 2]{BGRS97} (see also~\cite[Thm. 6.5]{EK15})\footnote{In~\cite[Thm. 6.5]{EK15} one needs to replace $f$ by $\frac{1}{d}f$ in the formulation (note a misprint, a missing bracket in eq.~(6.17) of the book).}
	that the lowest eigenvalue $\lm_1^{\rm D}(\Omg_\eps)$ of the Dirichlet Laplacian on $\Omg_\eps$ has the asymptotic expansion
	\begin{equation}\label{eq:BGRS}
		\lm_1^{\rm D}(\Omg_\eps) = \left(\frac{\pi}{d}\right)^2-\eps^2
		\left(\frac{\pi}{d}\right)^4\left(\frac{1}{d}\int_\dR f(x_1)\dd x_1\right)^2 + \cO(\eps^3),\qquad \eps\arr0.
	\end{equation}
	It follows from~\cite[Thm. S.14]{RSI} that $\aa + \sfH_{\aa,\eps}$ converges  as $\aa\arr+\infty$ to the Dirichlet Laplacian on $\Omg_\eps$ in the generalized strong resolvent sense~\cite[\S 9.3]{W00}. Hence
	one gets that $\lim_{\aa\arr\infty}(\aa+\lm_1^\aa(\eps)) = \lm_1^{\rm D}(\Omg_\eps)$. Although, it is not possible to derive from~\eqref{eq:asymptoticslm} the asymptotics in~\eqref{eq:BGRS} without extra information on the remainder, there is still a formal connection between these two asymptotic expansions.
	One can check easily that
	\[
	v_1(d) = \sqrt{\frac{2}{d}}\, \frac{\pi}{d\sqrt{\alpha}} + \OO(\alpha^{-1}),\qquad \mu_1 = -\alpha + \left(\frac{\pi}{d}\right)^2 + \OO(\alpha^{-1/2}),\qquad \alpha\to+\infty,
	\]
	and consequently,
	the first and the second terms in the asymptotic expansion of $\aa+\lm_1^\aa(\eps)$ in~\eqref{eq:asymptoticslm} (with respect to the small parameter $\eps$) converge to the respective terms in the asymptotic expansion of $\lm_1^{\rm D}(\Omg_\eps)$ in~\eqref{eq:BGRS} as $\aa\arr+\infty$.
\end{remark}
\begin{remark} \label{strong(ii)}
The claim (ii) in Theorem~\ref{thm} holds in the weak deformation regime only. To explain the point let us consider a
non-negative function $f_- \in C_0^\infty (\dR )$ (not identically equal to zero)  and the Hamiltonian $\sfH_{\aa, \eps}^-$ with the deformation defined by $\eps f_-$. It follows from Theorem~\ref{thm}\,(i) that there exists $\eps_0>0$ such that $\sfH_{\aa, \eps}^-$ admits a unique discrete eigenvalue for all $\eps \in (0, \eps_0]$. Relying on the min-max principle we conclude that the discrete spectrum is non-empty also for all $\eps> \eps_0$ since $(u,\sfH_{\aa, \eps }^-  u)_{L^2(\dR^2)}\leq ( u, \sfH_{\aa, \eps_0 }^- u)_{L^2(\dR^2)}$ for any $u\in H^2 (\dR^2 )$.
Consider now  a non-negative  function $f_+ \in C_0^\infty (\dR )$ with $\max f_{+}<d$ such that  $\int_{\dR} f_+(x_1) \,\dd x_1 > \int_{\dR} f_-(x_1) \,\dd x_1$  and the Hamiltonian $\sfH_{\aa, 1 }$ with the deformation defined by $f(x)=f_-(x_1) -f_+ (x_1 -\rho )$, $\rho >0$, which means that the assumption of (ii) in Theorem~\ref{thm} is satisfied.
Note that for all $u\in H^2 (\dR^2 )$ we have
\begin{equation}\label{eq:H_conv}
(u,\sfH_{\aa, 1 }^- u)_{L^2(\dR^2)}- ( u,\sfH_{\aa, 1 }u)_{L^2(\dR^2)} = -\alpha \int_{\Omg_0 \setminus  \Omg_1 (\rho )} |u|^2 \dd x\to 0 \,, \qquad \rho  \to \infty \,,
\end{equation}
where $\Omg_1 (\rho ):= \big\{(x_1,x_2)\in\dR^2\colon 0 < x_2 < d-  f_{+} (x_1 -\rho  )\big\}$.
Let $\psi \in H^2 (\dR^2)$ stand for the normalized ground state of  $\sfH_{\aa, 1 }^-$ corresponding to the eigenvalue $\mu<\mu_1$.
Then $(\psi, \sfH_{\aa,  1}\psi )_{L^2(\dR^2)}= \mu +\delta(\rho)$,  where $\delta(\rho)\to 0$ as $\rho\to\infty$ holds in view of \eqref{eq:H_conv}; this shows that the discrete spectrum of $\sfH_{\aa, 1 }$ is non-empty. 
\end{remark}

While the results for soft and Dirichlet waveguides correspond to each other, though, they are obtained by very different means. The analysis of the Dirichlet waveguide in \cite{BGRS97, EK15} relies on the `straightening' of $\Omg_\eps$ using an appropriate change of the coordinate system. Upon this geometric transformation the geometry of the system encoded in the function $f$ enters into coefficients in the differential expression of the Hamiltonian. In contrast, for the soft waveguides considered here we use another technique: to prove Theorem~\ref{thm} we employ a modification of the Birman-Schwinger principle in order to derive an implicit scalar equation for the lowest eigenvalue. The remaining analysis reduces then to the study of this equation.

As our second result we obtain an expansion of the eigenfunction of $\sfH_{\aa,\eps}$ associated to its lowest eigenvalue in the weak deformation limit of the regime $\int_\dR f(x_1)\,\dd x_1 > 0$. In order to formulate this result we introduce the function
$V_{\aa,\eps} = \sqrt{\aa}\big(\chi_{\Omg_\eps} - \chi_{\Omg_0}\big)$.
\begin{theorem}\label{thm:ef}
	Let $\eps \in (0,\eps_0]$ and let the self-adjoint operator $\sfH_{\aa,\eps}$ in the Hilbert space $L^2(\dR^2)$ be defined as in~\eqref{eq:operator}. Assume that $\int_\dR f(x_1)\,\dd x_1 > 0$ holds. For all sufficiently small $\eps > 0$ the eigenfunction corresponding to the unique simple eigenvalue $\lm_1^\aa(\eps)$ of $\sfH_{\aa,\eps}$ can be represented as
	\begin{equation}
	\label{eq:expansion_ef}
		\psi_\eps(x) = u_\eps(x) + v_\eps(x),
	\end{equation}
	where the leading contribution is
	\[
		u_\eps(x) := \frac{v_1(x_2)}{\sqrt{\eps}}\int_{\dR^2} e^{-\wh\dl(\eps)|x_1-x_1'|} v_1^2(x_2')V_{\aa,\eps}(x')\,\dd x'
		\qquad\text{with}\quad
		\wh\dl(\eps) := \eps\left(
		\frac{\aa v_1^2(d)}{2}
		\right)\int_{\dR}f(x_1)\,\dd x_1
	\]
	and the remainder  $v_\eps$ is given by \eqref{uv_def} below. Moreover, the norms of the two terms in the expansion~\eqref{eq:expansion_ef} are
    given by
	\[
		\|u_\eps\|_{L^2(\dR^2)} = \sqrt{2}v_1(d)\left[\int_\dR f(x_1)\,\dd x_1\right]^{1/2} + \cO(\sqrt{\eps})\qquad\text{and}\qquad
		\|v_\eps\|_{L^2(\dR^2)} = \cO(\sqrt{\eps}).
	\]
\end{theorem}
The proof of Theorem~\ref{thm:ef} relies on the Birman-Schwinger principle combined with Theorem~\ref{thm}\,(i). The main idea is to decompose the Birman-Schwinger operator in a convenient way; this leads to non-trivial technical estimates as we will see in Section~\ref{s:ef_proof} below.

In our last result we deal with the critical case, $\int_{\dR} f(x_1) \,\dd x_1 = 0$. We obtain a sufficient condition on the function $f$ in terms of coupling constant $\aa$ and strip width $d$ under which the bound state of $\sfH_{\aa,\eps}$ exists for all sufficiently small $\eps > 0$.
\begin{theorem}\label{thm:existence}
	Let $\eps \in (0,\eps_0]$ and let the self-adjoint operator $\sfH_{\aa,\eps}$ in the Hilbert space $L^2(\dR^2)$ be defined as in~\eqref{eq:operator}. Assume, in addition, that the compactly supported function $f$
 is such that  $f\in W^{1,\infty}(\dR)$ and that $\int_\dR f(x_1)\,\dd x_1 = 0$. If the condition
	\begin{equation}\label{eq:condition_existence}
		\frac{\int_{\dR}|f'(x_1)|^2\,\dd x_1}{\int_{\dR}|f(x_1)|^2\,\dd x_1} <
		\frac{\aa|v_1(d)|^2}{\sqrt{-\mu_1}}
	\end{equation}
is satisfied, then the operator $\sfH_{\aa,\eps}$ has a unique simple eigenvalue below the bottom of the essential spectrum for all sufficiently small $\eps > 0$.
\end{theorem}
The sufficient condition~\eqref{eq:condition_existence} is reminiscent of the one obtained in~\cite[Thm.~2]{EV97}, see also \cite[Thm.~6.9]{EK15}, for the Dirichlet Laplacian on $\Omg_\eps$ in the critical case and it shares with it the property that one can apply it to perturbations elongated enough: given a non-zero real-valued compactly supported function $f\in W^{1,\infty}(\dR)$ with $\int_{\dR} f(x_1)\,\dd x_1 = 0$ one can satisfy the condition~\eqref{eq:condition_existence} for perturbations described by $f_\gg(x) := f(\gg x)$ with a sufficiently small $\gg > 0$.
On the other hand, a comparison with the sufficient condition in~\cite{EV97} analogous to the one in Remark~\ref{rem:limit} is not possible here; note that the right-hand side of ~\eqref{eq:condition_existence} is of order of $\cO(\alpha^{-1/2})$ as $\aa\to+\infty$.

The proof of Theorem~\ref{thm:existence} given in Section~\ref{s:crit} is purely variational. We construct a suitable trial function, which depends on a parameter, taking inspiration in the trial function used in~\cite{EV97}; optimizing the obtained condition with respect to the parameter we get~\eqref{eq:condition_existence}. This result leaves some question open. It concerns not only a sufficient condition allowing for comparison with a critical Dirichlet strip, but also conditions ensuring the absence of the discrete spectrum for a fixed critical $f$ and large $d$ similar to what is known in the Dirichlet case~\cite{EV97}.

\setcounter{equation}{0}
\section{Preliminaries}

\subsection{Essential spectrum}

The essential spectrum of $\sfH_{\aa,\eps}$ can be determined using a compact perturbation argument. In the following we use the function
\[
	U_{\aa,\eps} := \aa\big(\chi_{\Omg_\eps} - \chi_{\Omg_0}\big).
\]
It is straightforward to see that $\sfH_{\aa,\eps} = \sfH_{\aa,0} - U_{\aa,\eps}$ and that $V_{\aa,\eps} = \sign(U_{\aa,\eps})|U_{\aa,\eps}|^{1/2}$.
\begin{proposition}\label{ess}
	Let $\eps\in [0,\eps_0]$ and let the self-adjoint operator $\sfH_{\aa,\eps}$ in the Hilbert space $L^2(\dR^2)$ be defined as in~\eqref{eq:operator}, then we have $\sess(\sfH_{\aa,\eps})= [\mu_1,\infty)$.
\end{proposition}
\begin{proof}
By \cite[Thm. 10.2]{Te}, $U_{\aa,\eps}$
being bounded and compactly supported is relatively compact with respect to the free Laplacian. Since $\sfH_{\aa,0}$ is a bounded perturbation of the free Laplacian, we conclude that $U_{\aa,\eps}$ is relatively compact with respect to $\sfH_{\aa,0}$ as well. Hence it follows by the stability of the essential spectrum under relatively compact perturbations that	
\[
\sess(\sfH_{\aa,\eps}) = \sess(\sfH_{\aa,0}) =
\s(\sfH_{\aa,0}) = [\mu_1,\infty).\qedhere
\]
\end{proof}

\subsection{Birman-Schwinger principle}
Let us turn to the discrete spectrum of $\sfH_{\aa,\eps}$. As usual, the spectral properties are encoded in its resolvent which we denote as
$$
\sfR_{\aa,\eps} (\kappa) := (\sfH_{\aa,\eps}+\kappa^2)^{-1}
$$
for $-\kappa^2\in \rho(\sfH_{\aa,\eps})$. An efficient tool to study the weak coupling behavior of the discrete spectrum is the Birman-Schwinger principle, that is, the following equivalence (see \eg~\cite[Lem. 1]{B95},
\cite[Prop. 7.9.2\,(b)]{S4})
\begin{equation}\label{eq-BS}
  -\kappa^2 \in \sigma_{\mathrm{d}} (\sfH_{\aa,\eps})\quad
  \Longleftrightarrow \quad \ker \big(\sfI-
 \sign(U_{\aa,\eps}) |U_{\aa,\eps}|^{1/2}\sfR_{\aa,0}(\kappa )|U_{\aa,\eps }|^{1/2}\big) \ne\varnothing,
\end{equation}
for all $\kp > \sqrt{-\mu_1}$;
moreover, the multiplicities of the eigenvalues on the both sides are the same. It is easy to see that
$\sign(U_{\aa,\eps}) |U_{\aa,\eps}|^{1/2}\sfR_{\aa,0}(\kappa )|U_{\aa,\eps }|^{1/2}$ is a compact operator in $L^2(\dR^2)$.
Moreover, this operator is non-negative provided that the profile function $f$ is non-negative.

\subsection{Decomposition of the free resolvent}

The key point for the further discussion is a particular decomposition of $\sfR_{\aa,0}(\kappa )$ which we will construct in this subsection. The operator $\sfH_{\aa,0}$ admits the following decomposition,
\[
\sfH_{\aa,0}= \ov{\sfh_{0}\otimes \sfI + \sfI\otimes \sfh_\aa},
\]
where $\sfh_0$ and $\sfh_\aa$ are defined as in~\eqref{eq-onedinHamil}.
Let us introduce the orthogonal projection $\sfp_0\psi := (\psi,v_1)_{L^2(\dR)}v_1$ in the Hilbert space $L^2(\dR)$ corresponding to the lowest eigenvalue of $\sfh_\aa$. We can naturally represent $\sfh_\aa$ as $\mu_1\sfp_0\oplus \sfh_\aa(\sfI-\sfp_0)$ with respect to $L^2(\dR) = \sfp_0(L^2(\dR))\oplus (\sfI-\sfp_0)(L^2(\dR))$. In its turn,
$\sfp_0$ induces the orthogonal projection $\sfP_0 := \sfI\otimes\sfp_0$ acting in the Hilbert space $L^2(\dR^2) = L^2(\dR)\otimes L^2(\dR)$ as
\[
	(\sfP_0 u)(x_1,x_2) = v_1(x_2)\int_\dR u(x_1,x_2')v_1(x_2')\,\dd x_2'
\]
and the Hilbert space $L^2(\dR^2)$ can be accordingly decomposed into the orthogonal sum
\[
L^2(\dR^2) = \sfP_0(L^2(\dR^2))\oplus (\sfI-\sfP_0)(L^2(\dR^2)).
\]
We get $\sfH_{\aa,0} = \sfH_{\aa,0}^I\oplus\sfH_{\aa,0}^{II}$ with respect to the above decomposition of $L^2(\dR^2)$ where
\[
\begin{aligned}
	\sfH_{\aa,0}^I = \sfh_0\otimes\sfp_0 + \sfI\otimes \mu_1\sfp_0\quad\text{and}\quad
	\sfH_{\aa,0}^{II} = \ov{\sfh_0\otimes(\sfI-\sfp_0) + \sfI\otimes (\sfh_\aa(\sfI-\sfp_0))}.
\end{aligned}
\]
We decompose the resolvent $\sfR_{\aa,0}(\kp)$ into the sum of two self-adjoint operators
\begin{equation}\label{eq:decomp_R}
	\sfR_{\aa,0}(\kp) = 	\sfR_{\aa,0}^I(\kp) +	\sfR_{\aa,0}^{II}(\kp),\quad\text{where} \quad
	\sfR_{\aa,0}^I(\kp)\! :=\! \sfR_{\aa,0}(\kp)\sfP_0,\quad
	\sfR_{\aa,0}^{II}(\kp)\! :=\! \sfR_{\aa,0}(\kp)(\sfI - \sfP_0).
\end{equation}
Using the expression in~\cite[eq. (7.47)]{Te} for the integral kernel of the resolvent of $\sfh_0$ we derive the representation of $\sfR^I_{\aa,0}(\kp)$ as an integral operator
\begin{equation}\label{eq:RI}
\begin{aligned}
	(\sfR^I_{\aa,0}(\kp)u)(x_1,x_2)& = \left(\left((\sfh_0+\kp^2+\mu_1)^{-1}\otimes \sfI\right)\sfP_0u\right)(x_1,x_2)\\
	&=
	\frac{1}{2\sqrt{\mu_1+\kp^2}}\int_\dR\int_\dR  e^{-\sqrt{\mu_1+\kp^2}|x_1-x_1'|}v_1(x_2)v_1(x_2')u(x_1',x_2')\,\dd x_1'\,\dd x_2'.
\end{aligned}
\end{equation}
Let us introduce the following new parameter $\delta := \sqrt{\kappa^2 +\mu_1} > 0$ for $\kp > \sqrt{-\mu_1}$. For the sake of convenience we also introduce the following shorthand notation,
\begin{equation}\label{eq:whR}
\begin{aligned}
	&\qquad\qquad\wh{\sfR}_{\aa,0}(\delta)  := \sfR_{\aa,0}(\sqrt{-\mu_1+\dl^2}),\\
	\wh{\sfR}_{\aa,0}^I(\delta) &:= \sfR_{\aa,0}^I(\sqrt{-\mu_1+\dl^2}),\qquad
	\wh{\sfR}_{\aa,0}^{II}(\delta) := \sfR_{\aa,0}^{II}(\sqrt{-\mu_1+\dl^2}).
\end{aligned}
\end{equation}

\subsection{Reformulation of the Birman-Schwinger principle}\label{ssec:ref_BS}

Our aim is to analyze the integral operator in the Birman-Schwinger principle~\eqref{eq-BS}. We begin by inspecting the term $\sign(U_{\aa,\eps})|U_{\aa,\eps} |^{1/2} \wh{\sfR}_{\aa,0}^{I}(\delta) |U_{\aa,\eps}|^{1/2}$ which is a bounded integral operator in $L^2(\dR^2)$ with the kernel
\[
	\cG_{\eps,\dl} (x,x')=
	\frac{1}{2\delta }\,\sign(U_{\aa,\eps}(x))
	|U_{\aa,\eps}(x) |^{1/2} \, e^{-\delta |x_1-x_1'|}\, v_1(x_2)v_1(x_2')\,|U_{\aa,\eps}(x')| ^{1/2}\,.
\]
where $x = (x_1,x_2)$ and $x'=(x_1',x_2')$.
We decompose it into a singular and regular part,
\begin{equation}\label{eq-I}
\sign(U_{\aa,\eps})|U_{\aa,\eps} |^{1/2} \wh\sfR_{\aa, 0 }^{I}(\dl) |U_{\aa,\eps}|^{1/2} = \sfL_{\eps,\dl} +\sfM_{\eps,\dl},
\end{equation}
where the integral kernels of $\sfL_{\eps,\delta} $  and $\sfM_{\eps,\delta}$ are of the form
\begin{equation}\label{eq-Lkernel}
\cL_{\eps,\delta} (x,x')=\frac{1}{2\delta}\,
\sign(U_{\aa,\eps}(x))|U_{\aa,\eps}(x) |^{1/2} v_1(x_2)v_1(x_2')\,|U_{\aa,\eps}(x')|^{1/2}
\end{equation}
and
\begin{equation}\label{eq-Mdel}
\cM_{\eps,\delta}(x,x')= \,\sign(U_{\aa,\eps}(x))|U_{\aa,\eps}(x)|^{1/2}\,
	m_{\dl}(x_1,x_1')
\, v_1(x_2)v_1(x_2')\,|U_{\aa,\eps}(x')| ^{1/2}\,,
\end{equation}
with
 \begin{equation} \label{eq:def_kernelM}
	m_{\dl}(x_1,x_1') := \frac{e^{-\dl|x_1-x_1'|}-1}{2
		\dl}\,.
\end{equation}
In view of \eqref{eq:decomp_R} and \eqref{eq-I} we are  able to decompose the total Birman-Schwinger operator as

\begin{equation}\label{eqdecomRLN}
\sign(U_{\aa,\eps})|U_{\aa,\eps} |^{1/2} \wh\sfR_{\aa,0}(\delta  )|U_{\aa,\eps}|^{1/2} =   \sfL_{\eps,\delta} +\sfN_{\eps,\delta},
\end{equation}
where
\begin{equation*}\label{eq-deN}
	\sfN_{\eps,\delta}
	:=
	\sfM_{\eps,\delta} +  \sign(U_{\aa,\eps})|U_{\aa,\eps} |^{1/2} \wh\sfR^{II}_{\aa, 0}(\delta )|U_{\aa,\eps}| ^{1/2}.
\end{equation*}
Note that it follows from~\eqref{eq-Lkernel} that $\sfL_{\eps,\dl}$ is a rank-one operator. 

We obtain in the lemma below norm estimates
for the non-negative bounded operator $\wh\sfR^{II}_{\aa,0}(\dl)$ and its square root. In the formulation of this lemma
we employ the continuous trace map $\G_{x_2}\colon H^1(\dR^2)\arr L^2(\dR)$
for $x_2\in\dR$ defined by $\G_{x_2} u := u|_{\dR\times\{x_2\}}$, 
where $u|_{\dR\times\{x_2\}} $ is the trace of $u$ on the straight line $\dR\times\{x_2\}$; (see~\cite[Chap. 3]{McL}). 
\begin{lemma}\label{lem:R_norm_est}
	For all $\dl > 0$ the following hold.
	\begin{myenum}
		\item There exists a constant $A_\aa > 0$ (independent of $\dl$) such that
		\begin{equation*}\label{key}
			\|\wh\sfR^{II}_{\aa,0}(\dl)\|_{L^2\arr L^\infty} \le A_\aa,
		\end{equation*}
		where $\|\cdot\|_{L^2\arr L^\infty}$ denotes the norm of an operator as a mapping from $L^2(\dR^2)$ into $L^\infty(\dR^2)$.
		\item 
		There exists a constant $B_\aa > 0$ (independent of $\dl$ and $x_2$) such that 	\begin{equation*}\label{key}
			\big\|\G_{x_2}
			\big(\wh\sfR^{II}_{\aa,0}(\dl))^{1/2}\big\|_{L^2(\dR^2)\arr L^2(\dR)}
			\le B_\aa,\qquad \forall\,x_2\in\dR.
		\end{equation*}
	\end{myenum}
\end{lemma}
\begin{proof}
\noindent (i)	
Recall that
\begin{equation}\label{eq:H2_norm}
\|u\|_{H^2(\dR^2)}^2 := \|u\|^2_{L^2(\dR^2)} + \|\Delta u\|^2_{L^2(\dR^2)}
\end{equation}
defines the norm in the Sobolev space $H^2(\dR^2)$. By continuity of the embedding of $H^2(\dR^2)$ into $L^\infty(\dR^2)$  (see \eg~\cite[Cor. 9.13]{Br}) there exists a constant $c >0$ such that
\begin{equation}\label{eq:embedding}
\|u\|_{L^\infty(\dR^2)} \le c\|u\|_{H^2(\dR^2)}
\end{equation}
for any $u\in H^2(\dR^2)$.
We aim at estimating the norm of $\wh\sfR^{II}_{\aa,0}(\dl)$ as a mapping from $L^2(\dR^2)$ into $L^\infty(\dR^2)$. For an arbitrary $u\in L^2(\dR^2)$ we get
\begin{equation*}\label{eq:resolvent_H2}
\begin{aligned}
\|\wh\sfR_{\aa,0}^{II}(\dl)u\|_{H^2(\dR^2)}^2&= \|-\Delta\wh\sfR_{\aa,0}^{II}(\dl)u\|_{L^2(\dR^2)}^2 +
\|\wh\sfR_{\aa,0}^{II}(\dl)u\|_{L^2(\dR^2)}^2\\
&=
\|-\Delta\wh\sfR_{\aa,0}(\dl)(\sfI-\sfP_0)u\|_{L^2(\dR^2)}^2 +
\|\wh\sfR_{\aa,0}(\dl)(\sfI-\sfP_0)u\|_{L^2(\dR^2)}^2\\
&=
\|(\sfH_{\aa,0}\! +\!\kp^2
\!+\!\aa\chi_{\Omg_0}\!-\!\kp^2)
\wh\sfR_{\aa,0}(\dl)(\sfI-\sfP_0)u\|^2_{L^2(\dR^2)}\\
&\qquad\qquad + \|\wh\sfR_{\aa,0}(\dl)(\sfI-\sfP_0)u\|^2_{L^2(\dR^2)}\\
&\le
2\|(\sfI-\sfP_0)u\|^2_{L^2(\dR^2)}
+
(2(\aa+\kp^2)^2+1)\|\wh\sfR_{\aa,0}(\dl)(\sfI-\sfP_0)u\|^2_{L^2(\dR^2)}\\
&\le
\left(2+ \frac{2(\aa+\kp^2)^2+1}{(\mu_2-\mu_1+\dl^2)^2}\right)
\|u\|^2_{L^2(\dR^2)}\\
&\le \left(2+\frac{2(\aa-\mu_1)^2+1}{(\mu_2-\mu_1)^2}\right)\|u\|^2_{L^2(\dR^2)},
\end{aligned}
\end{equation*}
where we used the expression~\eqref{eq:H2_norm} in the first step, employed in the second step the definition of $\wh\sfR^{II}_{\aa,0}$ given in~\eqref{eq:decomp_R}-\eqref{eq:whR}, performed simple algebraic manipulations in the third step, applied the triangle inequality for the norm in the fourth step, employed the spectral theorem in the penultimate step, and finally, used the fact that the pre-factor is maximal for $\dl =0$ in the last step. Combining the last estimate with~\eqref{eq:embedding} we obtain that
\begin{equation*}\label{eq:resLinf}
\|\wh\sfR^{II}_{\aa,0}(\dl)\|_{L^2\arr L^\infty} \le A_\aa,\qquad\text{for}\quad
A_\aa := 
c  \left(2+\frac{2(\aa-\mu_1)^2+1}{(\mu_2-\mu_1)^2}\right)^{1/2}.
\end{equation*}

\noindent (ii) It follows from the definitions of the operator $\wh\sfR^{II}_{\aa,0}(\dl)$ and the projection $\sfP_0$ that
\begin{equation}\label{eq:Rsquareroot}
	(\wh\sfR^{II}_{\aa,0}(\dl))^{1/2} 
	= 
	(\sfH_{\aa,0} -\mu_1+\dl^2)^{-1/2}
	(\sfI - \sfP_0).
\end{equation}
In particular, we conclude that $\ran(\wh\sfR^{II}_{\aa,0}(\dl))^{1/2}\subset H^1(\dR^2)$.
Recall that
\begin{equation*}\label{eq:H1_norm}
\|u\|_{H^1(\dR^2)}^2 := \|u\|^2_{L^2(\dR^2)} + \|\nabla u\|^2_{L^2(\dR^2;\dC^2)}
\end{equation*}
defines the norm in the Sobolev space $H^1(\dR^2)$. By continuity of the trace mapping $\G_{x_2}$ (see \eg~\cite[Lem. 3.35]{McL}) there exists a constant $c' >0$
such that
\begin{equation}\label{eq:trace}
\|\G_{x_2}u\|_{L^2(\dR)} \le c'\|u\|_{H^1(\dR^2)}
\end{equation}
for any $u\in H^1(\dR^2)$ and $x_2\in\dR$.
For any $u\in L^2(\dR^2)$ we find
\[
\begin{aligned}
	\big\|\big(\wh\sfR^{II}_{\aa,0}(\dl)\big)^{1/2}u\big\|_{H^1(\dR^2)}^2 & = 
	\big\|\nb \big(\wh\sfR^{II}_{\aa,0}(\dl)\big)^{1/2}u\big\|^2_{L^2(\dR^2;\dC^2)} + 
	\big\|\big(\wh\sfR^{II}_{\aa,0}(\dl)\big)^{1/2}u\big\|^2_{L^2(\dR^2)}\\
	&\le
	\frh_{\aa,0}\big[\big(\wh\sfR^{II}_{\aa,0}(\dl)\big)^{1/2}u\big]+ 
	(\aa+1)\big\|\big(\wh\sfR^{II}_{\aa,0}(\dl)\big)^{1/2}u\big\|^2_{L^2(\dR^2)},
\end{aligned}
\]
where the quadratic form $\frh_{\aa,0}$ is defined as in~\eqref{eq:frhaaeps}. Using the second representation theorem~\cite[Chap. VI, Thm. 2.23]{Kato} and the expression~\eqref{eq:Rsquareroot}
we obtain that
\[
\begin{aligned}
	\big\|\big(\wh\sfR^{II}_{\aa,0}(\dl)\big)^{1/2}u\big\|_{H^1(\dR^2)}^2&\le \Big\|
	(\sfH_{\aa,0} + \aa +1)^{1/2}(\sfH_{\aa,0} -\mu_1+\dl^2)^{-1/2}(\sfI - \sfP_0)u\Big\|^2_{L^2(\dR^2)}\\
	&\le 
	\left(\sup_{\lm\in[\mu_2,\infty)}\frac{\lm+\aa+1}{\lm-\mu_1+\dl^2}\right)\|u\|^2_{L^2(\dR^2)}\le \frac{\mu_2+\aa+1}{\mu_2-\mu_1}\|u\|^2_{L^2(\dR^2)},			
\end{aligned}
\]
where we applied the spectral theorem in between. It follows from the last estimate combined with~\eqref{eq:trace} that
\[
	\big\|\G_{x_2}\big(\wh\sfR^{II}_{\aa,0}(\dl)\big)^{1/2}\big\|_{L^2(\dR^2)\arr L^2(\dR)}\le B_\aa,\qquad \text{for}\quad B_\aa := c'\left(\frac{\mu_2+\aa+1}{\mu_2-\mu_1}\right)^{1/2}.\qedhere
\]
\end{proof}
In the next lemma we get an estimate on the norm of the operator $\sfN_{\eps,\dl}$ and analyze its dependence on $\eps$ and $\dl$.
\begin{lemma} \label{le-Nsmall}
For all $\eps,\dl >0$ the operator $\sfN_{\eps,\delta}\colon L^2(\dR^2)\to L^2 (\dR^2)$ is bounded and its norm satisfies $\|\sfN_{\eps,\delta}\| \leq C\|U_{\aa,\eps} \|_{L^1(\dR^2)}$ with a constant $C>0$ independent of $\delta$ and $\eps$. In particular,  $\|\sfN_{\eps,\delta}\| \to 0$ holds as $\eps\arr 0$ uniformly in $\dl$. Moreover, the operator-valued function $(0,\eps_0)\times(0,\infty)\ni(\eps,\dl)\mapsto\sfN_{\eps,\delta}$ is continuous in the operator norm.
\end{lemma}
\begin{proof}
\noindent\emph{Step 1.}	
In this step we estimate the norm of $\sfM_{\eps,\delta}$. To this aim, we note that using the inequality $1-e^{-x} \le x$ for $x\ge0$ we find for $(x_1,x_2),(x_1',x_2')\in\supp U_{\aa,\eps}$  the bound
\[
	|m_{\dl}(x_1,x_1')| \le  M\,,
\]
where
\begin{equation}\label{eq:def_Mprime}
M :=\frac12\sup_{x_1,x'_1\in\supp f}|x_1-x_1'|>0\,.
\end{equation}
Therefore, the  kernel of $\sfM_{\eps,\delta}$ given by \eqref{eq-Mdel} admits the pointwise estimate
\[
	|\cM_{\eps,\delta}(x,x')| \leq M\, |U_{\aa,\eps}(x) |^{1/2}\, v_1(x_2)v_1(x_2')\, |U_{\aa,\eps}(x') |^{1/2}.
\]
Since the support of $U_{\aa,\eps}$ is  bounded uniformly for all sufficiently small $\eps >0$, we obtain that
\begin{equation}\label{eq:est_norm1}
\begin{aligned}
\|\sfM_{\eps,\dl}\| &\le
\|\sfM_{\eps,\dl}\|_\mathrm{HS}
\le M\left( \int_{\dR^2}\int_{\dR^2}
|U_{\aa,\eps}(x)|  v_1^2(x_2)v_1^2(x_2')|U_{\aa,\eps}(x')|\,\dd x \,\dd x'\right)^{1/2}\\
&\le M  \|U_{\aa,\eps}\|_{L^1(\dR^2)}\|v_1\|^2_{
	L^\infty(\dR)};
\end{aligned}
\end{equation} 
 here $\|\cdot\|_{\rm HS}$ stands for the Hilbert-Schmidt norm.

\smallskip

\noindent\emph{Step 2.}
Our next aim is to estimate the norm of $\sign(U_{\aa,\eps})|U_{\aa,\eps} |^{1/2} \wh\sfR^{II}_{\aa, 0}(\delta )|U_{\aa,\eps}| ^{1/2}$. 
Using Lemma~\ref{lem:R_norm_est}\,(ii) we get  
\begin{equation}\label{eq:est_norm2}
\begin{aligned}
	&\big\|\sign(U_{\aa,\eps})|U_{\aa,\eps}|^{1/2}\wh\sfR^{II}_{\aa,0}(\dl)|U_{\aa,\eps}|^{1/2}\big\| \le
	\big\||U_{\aa,\eps}|^{1/2}(\wh\sfR^{II}_{\aa,0}(\dl))^{1/2}\big\|^2\\
	&\qquad	 \le \aa\int_{d+\eps\min f}^{d+\eps\max f}\big\|\G_{x_2}(\wh\sfR^{II}_{\aa,0}(\dl))^{1/2})\big\|_{L^2(\dR^2)\arr L^2(\dR)}^2\dd x_2\\
	&\qquad\le \aa\eps(\max f-\min f)B_\aa^2
	= \frac{\max f-\min f}{\|f\|_{L^1(\dR)}}B_\aa^2
		\|U_{\aa,\eps}\|_{L^1(\dR^2)}.
\end{aligned}
\end{equation}
The upper bound on the norm of $\sfN_{\eps,\dl}$ follows from~\eqref{eq:est_norm1} and~\eqref{eq:est_norm2} combined with the triangle inequality.

\smallskip

\noindent{\it Step 3.} In this step we will show that $\sfN_{\eps,\dl}$ is continuous in the operator norm as a function of $\eps$ and $\dl$. Continuity with respect to $\dl$ follows from the representation
\[
	\sfN_{\eps,\dl} = \sign(U_{\aa,\eps})|U_{\aa,\eps}|^{1/2}\wh\sfR_{\aa,0}(\dl)|U_{\aa,\eps}|^{1/2}-\sfL_{\eps,\dl}
\]
combined with the continuity in $\dl$ of $\sfL_{\eps,\dl}$ and of the resolvent $\wh\sfR_{\aa,0}(\dl)$ in the operator norm.

Continuity with respect to $\eps$ is more subtle. Let us take $\eps_1, \eps_2\in(0,\eps_0)$; our aim is to show that $\|\sfN_{\eps_1,\dl}-\sfN_{\eps_2,\dl}\|\arr0$ holds as $\eps_1 \arr\eps_2$. We use the notation
\[
\begin{aligned}
V_{\aa,\eps_1,\eps_2}(x)& :=
\sign(U_{\aa,\eps_1}(x))
|U_{\aa,\eps_1}(x)|^{1/2}-	\sign(U_{\aa,\eps_2}(x))
|U_{\aa,\eps_2}(x)|^{1/2},\\
W_{\aa,\eps_1,\eps_2}(x) & :=
|U_{\aa,\eps_1}(x)|^{1/2} -	
|U_{\aa,\eps_2}(x)|^{1/2};
\end{aligned}
\]
by means of the triangle inequality for the operator norm we then get
\begin{equation}\label{eq:Neps1eps2}
\begin{aligned}
	\|\sfN_{\eps_1,\dl}-\sfN_{\eps_2,\dl}\|
	&\le \|\sfM_{\eps_1,\dl} - \sfM_{\eps_2,\dl}\| +
	\Big\|
	\sign(U_{\aa,\eps_2})|U_{\aa,\eps_2}|^{1/2}\wh\sfR^{II}_{\aa,0}(\dl)W_{\aa,\eps_1,\eps_2}\Big\|\\
	&\qquad+\Big\|V_{\aa,\eps_1,\eps_2}\wh\sfR^{II}_{\aa,0}(\dl)|U_{\aa,\eps_1}|^{1/2}\Big\|.
	\\
\end{aligned}
\end{equation}
Let us recall the notation $V_{\aa,\eps} =\sign(U_{\aa,\eps})|U_{\aa,\eps}|^{1/2}$. Estimating the norm of the difference $\sfM_{\eps_1,\dl}-\sfM_{\eps_2,\dl}$ from above by its Hilbert-Schmidt norm we get
\begin{equation}\label{eq:Meps1eps2}
\begin{aligned}
	\|\sfM_{\eps_1,\dl}\! -\! \sfM_{\eps_2,\dl}\|^2&\!\le\!
		\|\sfM_{\eps_1,\dl} - \sfM_{\eps_2,\dl}\|^2_{\rm HS}\\
	&\!\le\! M^2\int_{\dR^2}\!\int_{\dR^2}\!
	\left|V_{\aa,\eps_1}(x)|V_{\aa,\eps_1}(x')|\!-\!V_{\aa,\eps_2}(x)|V_{\aa,\eps_2}(x')|\right|^2\!
	v_1^2(x_2)v_1^2(x_2')\,\dd x\,\dd x'\\
	&\le 4M^2\aa^2\|v_1\|_{L^\infty ( \dR )}^4\big
	|(\Omg_{\eps_1}\times\Omg_{\eps_1})\triangle(\Omg_{\eps_2}\times\Omg_{\eps_2})\big| \arr 0,\qquad \text{as}\,\,\eps_2\arr\eps_1,
\end{aligned}
\end{equation}
where $\cA\triangle \cB = (\cA\setminus \cB)\cup(\cB\setminus \cA)$ is the symmetric difference for open sets $\cA,\cB$ and $M$ is as in~\eqref{eq:def_Mprime}.
Using Lemma~\ref{lem:R_norm_est}\,(i) we get
\begin{equation}\label{eq:I}
\begin{aligned}
	\cI_{\eps_1,\eps_2} &:=\Big\|
	\sign(U_{\aa,\eps_2})|U_{\aa,\eps_2}|^{1/2}\wh\sfR^{II}_{\aa,0}(\dl)
	W_{\aa,\eps_1,\eps_2}\Big\|\\
	&=
	\Big\|
	W_{\aa,\eps_1,\eps_2}
	\wh\sfR^{II}_{\aa,0}(\dl)\sign(U_{\aa,\eps_2})|U_{\aa,\eps_2}|^{1/2}\Big\|\\
	&\le \|W_{\aa,\eps_1,\eps_2}\|_{L^2(\dR^2)}\cdot\|\wh\sfR^{II}_{\aa,0}(\dl)\|_{L^2\arr L^\infty  }
\cdot \| \left|U_{\aa,\eps_2}\right| ^{1/2}\|_{L^\infty (\dR^2)}\\
	&\le
	\aa |\Omg_{\eps_1}\triangle\Omg_{\eps_2}|^{1/2}  A_\aa\arr 0 \qquad\text{as}\;\eps_2\arr\eps_1.
\end{aligned}
\end{equation}
In a similar way, we obtain the following estimate
\begin{equation}\label{eq:J}
\begin{aligned}
	\cJ_{\eps_1,\eps_2} &:=\Big\|
	V_{\aa,\eps_1,\eps_2}\wh\sfR^{II}_{\aa,0}(\dl)|U_{\aa,\eps_1}|^{1/2}\Big\|\\
	&\le \|V_{\aa,\eps_1,\eps_2}\|_{L^2(\dR^2)}\cdot\|\wh\sfR^{II}_{\aa,0}(\dl)\|_{L^2\arr L^\infty}\cdot\||U_{\aa,\eps_1}|^{1/2}\|_{L^\infty(\dR^2)}\\
	&\le
	2\aa|\Omg_{\eps_1}\triangle\Omg_{\eps_2}|^{1/2} A_\aa \arr
	0\qquad\text{as}\;\eps_2\arr\eps_1.
\end{aligned}
\end{equation}
Combining~\eqref{eq:Neps1eps2} with~\eqref{eq:Meps1eps2},~\eqref{eq:I},~\eqref{eq:J} we conclude that $\sfN_{\eps,\dl}$ is continuous in the operator norm with respect to $\eps$.
\end{proof}
In the next proposition we show that if for all sufficiently small $\eps > 0$ the discrete spectrum of $\sfH_{\aa,\eps}$ is non-empty for a certain profile function $f$, then this discrete spectrum necessarily consists of a unique simple eigenvalue.
\begin{proposition}\label{prop:bs_bnd}
	The number of the eigenvalues $N_{\aa}(\eps)$, with multiplicities taken into account, of the operator $\sfH_{\aa,\eps}$ lying in the interval $(-\infty,\mu_1)$ satisfies the bound $N_\aa(\eps) \le 1$ for all sufficiently small $\eps > 0$.
\end{proposition}
\begin{proof}
	Without loss of generality we may assume that the profile function  $f$ is non-negative. Should $f$ be sign-changing we can replace $f$ by a non-negative profile function $g:=\max\{f,0\}$. Upon such a replacement, the essential spectrum of $\sfH_{\aa,\eps}$ remains the same, but the modified operator becomes smaller in the form sense, and hence in view of the min-max principle the value $N_{\aa}(\eps)$ can not  decrease.
	
	By~\cite[Thm. 7.9.4]{S4} (see also \cite[Eq. (1.14)]{FS11}) we obtain that the dimension of the spectral subspace of the operator $\sfH_{\aa,\eps}$ corresponding to the interval $(-\infty,\mu_1-\dl^2)$ with $\dl >0$ is equal to the dimension $n_{\aa,\dl}(\eps)$ of the spectral subspace of the self-adjoint operator $\sfL_{\eps,\dl}+\sfN_{\eps,\dl}$ corresponding to the interval $(1,\infty)$. For all sufficiently small $\eps >0$ we have by Lemma~\ref{le-Nsmall} that $\|\sfN_{\eps,\dl}\| < 1$ for any $\dl > 0$. Hence we can conclude from the facts that $\sfN_{\eps,\dl}$ is self-adjoint that $\sfL_{\eps,\dl}$ is a self-adjoint rank-one operator combined with the perturbation result~\cite[\S 9.3, Thm. 3]{BS87} that $n_{\aa,\dl}(\eps) \le 1$ for all sufficiently small $\eps >0$ and any $\dl > 0$. In this way, we obtain that $N_\aa(\eps) \le 1$.
\end{proof}

In the next lemma we reformulate the Birman-Schwinger principle \eqref{eq-BS} in a more convenient form for all sufficiently small $\eps > 0$. Moreover, we use this new formulation to derive a scalar equation that the lowest eigenvalue must satisfy.
\begin{lemma} \label{lem:BS}
For all sufficiently small $\eps > 0$
\begin{equation*}\label{eq-leq1}
\dim\ker(\sfH_{\aa,\eps} -\mu_1+\dl^2)
=
\dim\ker(\sfI-(\sfI-\sfN_{\eps,\delta})^{-1} \sfL_{\eps,\delta})\leq 1,
\end{equation*}
and,
\begin{equation}\label{eq-kerempty}
\ker(\sfI-(\sfI-\sfN_{\eps,\delta})^{-1} \sfL_{\eps,\delta}) \neq\{0\}
\end{equation}
holds if and only if
\begin{equation}\label{eq-kerfun}
F(\eps, \delta )=1
\end{equation}
with
$$
F(\eps, \delta ):= \frac{1}{2\delta } \int_{\dR^2}|U_{\aa,\eps}(x)|^{1/2}v_1 (x_2)
\big((\sfI-\sfN_{\eps,\delta}  )^{-1}
\sign(U_{\aa,\eps})  |U_{\aa,\eps} |^{1/2} v_1 \big) (x )\,\,\dd x;
$$
where we interpret the second entry of $v_1$ as the function $\dR^2\ni(x_1,x_2)\mapsto v_1(x_2)$. Moreover, the function $F$ is continuous in $\eps$ and $\dl$ for sufficiently small $\eps >0$ and for $\dl >0$.
\end{lemma}
\begin{proof}
It follows from the decomposition \eqref{eqdecomRLN} in combination with the Birman-Schwinger principle \eqref{eq-BS} that
\[
\dim\ker(\sfH_{\aa,\eps}-\mu_1+\dl^2) = \dim\ker(\sfI - \sfL_{\eps,\delta} -\sfN_{\eps,\delta}  ).
\]
By Lemma~\ref{le-Nsmall}
the operator $\sfI-\sfN_{\eps,\delta} $ is invertible for all $\eps>0$ small enough, so that
\[
	\ker \big(\sfI - \sfN_{\eps,\delta} -\sfL_{\eps,\delta}\big)=
	\ker\big(\sfI - (\sfI - \sfN_{\eps,\delta})^{-1} \sfL_{\eps,\delta}\big)
\]
and we have restated the problem of identifying discrete eigenvalues of $\sfH_{\aa,\eps}$ to the analysis of $\ker(\sfI - (\sfI - \sfN_{\eps,\delta})^{-1} \sfL_{\eps,\delta})$.

Note that  $\sfL_{\eps,\delta} $ is by \eqref{eq-Lkernel} a rank-one operator, and therefore the same holds for $(\sfI-\sfN_{\eps,\delta}  )^{-1} \sfL_{\eps,\delta}$. Consequently, the operator $(\sfI-\sfN_{\eps,\delta})^{-1} \sfL_{\eps,\delta}$ has one non-zero eigenvalue of multiplicity one, which yields
$\dim\ker(\sfI-(\sfI-\sfN_{\eps,\dl})^{-1}\sfL_{\eps,\dl})\le1$. Using \eqref{eq-Lkernel} again, we conclude that the corresponding eigenfunction is (a multiple of) $\phi_0:=(\sfI-\sfN_{\eps,\dl})^{-1}(\sign(U_{\aa,\eps})|U_{\aa,\eps} |^{1/2} v_1) $. Furthermore, we have
\begin{align*}
(\sfI-\sfN_{\eps,\delta})^{-1} \sfL_{\eps,\delta}\phi_0 = \frac{1}{2\delta }
\Big( \int_{\dR^2}|U_{\aa,\eps}(x)|^{1/2} v_1 (x_2)
\big((\sfI\!-\!\sfN_{\eps,\delta})^{-1}  (\sign(U_{\aa,\eps})|U_{\aa,\eps}|^{1/2} v_1) \big) (x)\, \,\dd x\Big) \phi_0.
\end{align*}
The condition \eqref{eq-kerempty} is satisfied if the mentioned eigenvalue equals one, in other words, if equation~\eqref{eq-kerfun} has a solution.

The function $F$ can be viewed as
\begin{equation}\label{eq:Frepr}
	F(\eps,\dl) = \frac{1}{2\dl}
	\left(|U_{\aa,\eps}|^{1/2}v_1, (\sfI -\sfN_{\eps,\dl})^{-1}(
	\sign(U_{\aa,\eps})|U_{\aa,\eps}|^{1/2} v_1)\right)_{L^2(\dR^2)}.
\end{equation}
In view of Lemma~\ref{le-Nsmall} the operator-valued function $(\sfI-\sfN_{\eps,\dl})^{-1}$ is continuous with respect to $\eps$ and $\dl$ in the operator norm for all $\eps > 0$ sufficiently small. The functions $|U_{\aa,\eps}|^{1/2}v_1$ and $\sign(U_{\aa,\eps})|U_{\aa,\eps}|^{1/2}v_1$ are obviously continuous in the norm of $L^2(\dR^2)$ with respect to variation of the parameter $\eps$. Consequently, it follows from the representation~\eqref{eq:Frepr} that the function $F$ is continuous in $\eps$ and $\dl$ for $\dl >0$ and sufficiently small $\eps$.
\end{proof}
\section{Proof of Theorem~\ref{thm}}

Now we can analyze the spectral equation \eqref{eq-kerfun}. Using Lemma~\ref{le-Nsmall} we can expand
for all sufficiently small $\eps > 0$ the inverse into the Neumann series, $(\sfI-\sfN_{\eps,\delta})^{-1}=\sfI+\sfN_{\eps,\delta} +\sfN_{\eps,\delta}^2+\cdots$, which
allows us to write the function $F$ in the form of a series,
$$
F(\eps, \delta ) = \sum_{j=0}^\infty \frac{1}{2\delta }\Big( \int_{\dR^2}|U_{\aa,\eps}(x)|^{1/2} v_1 (x_2 )
\big((\sfN_{\eps,\delta})^j (\sign(U_{\aa,\eps})|U_{\aa,\eps} |^{1/2} v_1)\big) (x)\,\,\dd x\Big)\,.
$$
 For $j\geq 1$ we can estimate the integral in the bracket using Lemma~\ref{le-Nsmall} and the Cauchy-Schwarz inequality
\[
\begin{aligned}
&\left|\int_{\dR^2}|U_{\aa,\eps}(x)|^{1/2} v_1 (x_2)
\big((\sfN_{\eps,\dl})^j \sign(U_{
	\aa,\eps})|U_{\aa,\eps} |^{1/2} v_1 \big) (x )\,\,\dd x \right| \\
&\qquad
\leq  \|(\sfN_{\eps,\delta}  )^j\|\cdot \||U_{\aa,\eps}|^{1/2}v_1 \|^2_{L^2(\dR^2)}\\[.3em] \label{eq-estim4}
  & \qquad \leq C^j\|v_1\|_{L^\infty(\dR^2)}^2 \|U_{\aa,\eps} \|^{j+1}_{L^1(\dR^2)}.
\end{aligned}
\]
Let us introduce the function $G(\eps,\dl) := \dl-\dl F(\eps,\dl)$.
We get from Lemma~\ref{lem:BS} that $\mu_1-\dl^2$ is an eigenvalue of $\sfH_{\aa,\eps}$ if, and only if $G(\eps,\dl) = 0$.
Observe that $G$ is continuous for $\dl > 0$ and for sufficiently small $\eps > 0$.
The above estimates allow us to write
\begin{equation*}\label{eq-spect3}
	G(\eps, \dl) = \dl - \frac12\int_{\dR^2} U_{\aa,\eps}(x) v_1^2(x_2)\,\dd x + \cO_{\rm u}\big(\|U_{\aa,\eps}\|_{L^1(\dR^2)}^2\big),\qquad \eps\arr0,
\end{equation*}
where $\cO_{\rm u}\big(g(\eps)\big)$ stands for a function in $\eps$ and $\dl$ that can be bounded from above by $|g(\eps)|$ multiplied by a positive constant, which is independent of $\dl$.
Using that $v_1\in H^2(\dR)$ and that $H^2(\dR)$ is continuously embedded into $C^1(\dR)$
(see \eg~\cite[Cor. 9.13]{Br}) we get $v_1(d+\eps') = v_1(d) + \cO(\eps')$ as $\eps'\arr 0$. In this way, we arrive at the expansion
\begin{equation}\label{eq:G_expansion}
\begin{aligned}
		G(\eps, \dl)
		&=
		\dl - \frac{\aa v_1^2(d)}{2}
		\big(|\Omg_\eps\sm\Omg_0| - |\Omg_0\sm\Omg_\eps|\big) + \cO_{\rm u}\big(\eps^2\big)\\
		& = \dl - \frac{\aa\eps v_1^2(d)}{2}\int_{\dR} f(x_1)\,\dd x_1 + \cO_{\rm u}\big(\eps^2\big),\qquad\eps\arr0.
\end{aligned}
\end{equation}
This asymptotics is equivalent to the fact that there exists a constant $c' > 0$ independent of $\dl > 0$  and such that
\[
	\dl - \frac{\aa\eps v_1^2(d)}{2}\int_{\dR} f(x_1)\,\dd x_1 - c'\eps^2\le G(\eps,\dl) \le \dl - \frac{\aa\eps v_1^2(d)}{2}\int_{\dR} f(x_1)\,\dd x_1 + c'\eps^2.
\]
Assume that $\int_{\dR} f(x_1)\,\dd x_1 > 0$. Then for sufficiently small $\eps > 0$ we define
\[
\dl_\pm(\eps) = \frac{\aa\eps v_1^2(d)}{2}\int_\dR f(x_1)\,\dd x_1 \pm 2c'\eps^2 > 0
\]
and	 we have $G(\eps,\dl_+(\eps)) > 0$ and $G(\eps,\dl_-(\eps)) < 0$. Hence by continuity of $G$ with respect to $\dl$ we get that there exists $\dl(\eps) \in (\dl_-(\eps),\dl_+(\eps))$  such that $G(\eps,\dl(\eps)) = 0$ and it admits the asymptotic expansion
\begin{equation}\label{eq:asymp_dl_eps}
	\dl(\eps) = \frac{\aa\eps v_1^2(d)}{2}\int_\dR f(x_1)\,\dd x_1 + \cO(\eps^2),\qquad \eps\arr0.
\end{equation}
As a result we get from Proposition~\ref{prop:bs_bnd} and Lemma~\ref{lem:BS} that for all sufficiently small $\eps >0$ there is a unique simple eigenvalue $\lm_1^\aa(\eps) = \mu_1-\dl(\eps)^2$ of $\sfH_{\aa,\eps}$ below the threshold of the essential spectrum and this eigenvalue admits the expansion
\[
	\lm_1^\aa(\eps) = \mu_1 -
	\eps^2
	\left(\frac{\aa^2 v_1^4(d)}{4}\right)
	\left(\int_\dR f(x_1)\,\dd x_1\right)^2 + \cO(\eps^3),\qquad \eps \arr 0,
\]
Thus, the claim of (i) is proved.
 
In the case that $\int_\dR f(x_1)\,\dd x_1 < 0$ we immediately conclude from~\eqref{eq:G_expansion} that for any sufficiently small $\eps > 0$ there is no $\dl > 0$ such that $G(\eps,\dl) = 0$. By Lemma~\ref{lem:BS} the operator $\sfH_{\aa,\eps}$ has then no eigenvalues below $\mu_1$ for all sufficiently small $\eps > 0$ and hence the claim of (ii) is proved as well. 
\section{Proof of Theorem~\ref{thm:ef}}
\label{s:ef_proof}

Recall that by Theorem~\ref{thm}(i) under the assumption $\int_\dR f(x_1)\,\dd x_1 > 0$ the discrete spectrum of $\sfH_{\aa,\eps}$ consists of a unique simple eigenvalue $\lm_1^\aa(\eps) < \mu_1$ for all sufficiently small $\eps >0$. Let $\dl(\eps) > 0$ be such that $\lm_1^\aa(\eps) = \mu_1 - \dl(\eps)^2$ holds as in the proof of Theorem~\ref{thm}. For the sake of brevity we use the notation $V_{\aa,\eps} = \sign(U_{\aa,\eps})|U_{\aa,\eps}|^{1/2}$ and $\sfL_\eps := \sfL_{\eps,\dl(\eps)}, \sfN_{\eps} := \sfN_{\eps,\dl(\eps)}$, $\sfM_\eps := \sfM_{\eps,\dl(\eps)}$ where the operator-valued functions $\sfL_{\eps,\dl}$, $\sfN_{\eps,\dl}$ and $\sfM_{\eps,\dl}$ are defined as in the beginning of Subsection~\ref{ssec:ref_BS}. In the course of the proof $\eps > 0$ is assumed to be sufficiently small.

\smallskip

{\noindent{\it Step 1.}}
Let us pick a non-trivial real-valued function $\phi_\eps \in \ker(\sfI - V_{\aa,\eps}\wh{\sfR}_{\aa,0}(\dl(\eps))|V_{\aa,\eps}|)$, which exists by the Birman-Schwinger principle~\eqref{eq-BS}. According to~\cite[Lem. 1]{B95},
%
\begin{equation}\label{eq:efunction}
	f_\eps := \wh\sfR_{\aa,0}(\dl(\eps))\phi_\eps
\end{equation}
is an eigenfunction of $\sfH_{\aa,\eps}$ corresponding to the eigenvalue $\lm_1^\aa(\eps) = \mu_1 - \dl(\eps)^2$. Using Lemma~\ref{le-Nsmall} we get that $\|\sfN_\eps\| \arr 0$ as $\eps \arr 0$. Hence in view of the decomposition performed in Subsection~\ref{ssec:ref_BS} we equivalently have that
\[
	\phi_\eps\in \ker\big(\sfI - (\sfI-\sfN_\eps)^{-1}\sfL_\eps\big).
\]
Relying on the expansion of $(\sfI-\sfN_\eps)^{-1} = \sfI + (\sfI-\sfN_\eps)^{-1}\sfN_\eps$ we get that
\begin{equation}\label{eq:expansion_psi}
	\phi_\eps = \left(\sfI + \sfN_\eps(\sfI-\sfN_\eps)^{-1}\right)\sfL_\eps\phi_\eps,
\end{equation}
where one has
\[
\begin{aligned}
	&(\sfL_\eps\phi_\eps)(x) = \frac{C_{\phi_\eps}}{2\dl(\eps)} \omg_{\aa,\eps}(x),\\
	&\qquad\qquad\text{with}\qquad C_{\phi_\eps}:= \int_{\dR^2}|V_{\aa,\eps}(x)|v_1(x_2) \phi_\eps(x)\,\dd x\quad\text{and}\quad \omg_{\aa,\eps}(x) := V_{\aa,\eps}(x)v_1(x_2).
\end{aligned}
\]
Substituting~\eqref{eq:expansion_psi} into~\eqref{eq:efunction} we get
\begin{equation}
\begin{aligned}
	f_\eps & = \wh\sfR_{\aa,0}(\dl(\eps))\phi_\eps
	=
	\big[
	\wh\sfR^{I}_{\aa,0}(\dl(\eps))
	+
	\wh\sfR^{II}_{\aa,0}(\dl(\eps))
	\big]\phi_\eps\\
	&=
	\frac{C_{\phi_\eps}}{2\dl(\eps)}\Big\{
	\wh\sfR^{I}_{\aa,0}(\dl(\eps))\omg_{\aa,\eps} +
		\wh\sfR^{I}_{\aa,0}(\dl(\eps))	\sfN_\eps(\sfI-\sfN_\eps)^{-1}\omg_{\aa,\eps}
	 + \wh\sfR^{II}_{\aa,0}(\dl(\eps))
	(\sfI-\sfN_\eps)^{-1}\omg_{\aa,\eps}
	\Big\},
\end{aligned}	
\end{equation}
where the operator-valued functions $\wh{\sfR}^I_{\aa,0}$ and $\wh{\sfR}^{II}_{\aa,0}$ are defined as in \eqref{eq:whR}. We may drop the constant factor $\frac{C_{\phi_\eps}}{2\dl(\eps)}$ by changing the normalization and  consider the eigenfunction in the form
\begin{equation}\label{eq:f_expansion}
	\psi_\eps := \wh\sfR^{I}_{\aa,0}(\dl(\eps))\omg_{\aa,\eps} +
	\wh\sfR^{I}_{\aa,0}(\dl(\eps))	\sfN_\eps(\sfI-\sfN_\eps)^{-1}\omg_{\aa,\eps}
	+ \wh\sfR^{II}_{\aa,0}(\dl(\eps))
	(\sfI-\sfN_\eps)^{-1}\omg_{\aa,\eps}.
\end{equation}
The remaining analysis reduces to separate consideration of the three terms at the right-hand side of \eqref{eq:f_expansion} which we denote as $a_\eps,b_\eps,c_\eps\in L^2(\dR^2)$,
\begin{equation}\label{eq:abc}
a_{\eps} \!:=\!
\wh\sfR^{I}_{\aa,0}(\dl(\eps))\omg_{\aa,\eps},\quad b_\eps\! :=\!	\wh\sfR^{I}_{\aa,0}(\dl(\eps))	 \sfN_\eps(\sfI-\sfN_\eps)^{-1}\omg_{\aa,\eps},\quad
c_\eps\! :=\! \wh\sfR^{II}_{\aa,0}(\dl(\eps))(\sfI-\sfN_\eps)^{-1}\omg_{\aa,\eps}.
\end{equation}

\smallskip

\noindent{\emph{Step 2.}}
In this step we show an auxiliary asymptotic expansion, special cases of which will be used in the next step of the proof in the estimates for the quantities~\eqref{eq:abc}. Let $(g_\eps)_\eps$ be a family of arbitrary functions $g_\eps\in L^2_{\mathrm{loc}}(\dR^2)$.
Our aim is to show that the norm of $h_\eps := \wh\sfR_{\aa,0}^I(\dl(\eps))(V_{\aa,\eps} g_\eps)$ has the asymptotic expansion
\begin{equation}\label{eq:norm_expansion}
	\|h_\eps\|^2_{L^2(\dR^2) } = \frac{v_1^2(d)(1+\cO(\eps))}{4\dl(\eps)^3}\left|
	\int_{\dR^2} V_{\aa,\eps}(x) g_\eps(x)\,\dd x
	\right|^2,\qquad \eps\arr 0.
\end{equation}
We remark that although the family of functions $g_\eps$ is not assumed to be in $L^2(\dR^2)$ still the function $h_\eps$ is well defined if we interpret it as the operator $\wh\sfR^I_{\aa,0}(\dl(\eps))$ applied to the product $V_{\aa,\eps}g_\eps$, which clearly belongs to the Hilbert space $L^2(\dR^2)$ and also to the Banach space $L^1(\dR^2)$. Using the definition of $\wh\sfR^I_{\aa,0}$ and the formula~\eqref{eq:RI} we find that
\[
	h_\eps(x) = \frac{v_1(x_2)}{2\dl(\eps)}
	\int_{\dR^2} e^{-\dl(\eps)|x_1-x_1'|} v_1(x_2')V_{\aa,\eps}(x') g_\eps(x')\,\dd x'.
\]
Using the fact that $v_1$ is normalized in $L^2(\dR)$ and replacing $v_1(x)$ in the neighbourhood of the point $x = d$ by the expansion $v_1(x) = v_1(d) + \cO(|x-d|)$ as $x\arr d$, we get
\[
	\|h_\eps\|^2 _{L^2(\dR^2) }= \frac{v_1^2(d)(1+\cO(\eps))}{4\dl(\eps)^2}\int_\dR\left|\int_{\dR^2} e^{-\dl(\eps)|x_1-x_1'|} V_{\aa,\eps}(x')g_\eps(x')\,\dd x'\right|^2\,\dd x_1.
\]
Performing the substitution $t = \dl(\eps) x_1$ in the outer integral we can rewrite the above formula as
\[
\|h_\eps\|^2 _{L^2(\dR^2) }= \frac{v_1^2(d)(1+\cO(\eps))}{4\dl(\eps)^3}\int_\dR\left|\int_{\dR^2} e^{-|t- \dl(\eps)x_1'|} V_{\aa,\eps}(x')g_\eps(x')\,\dd x'\right|^2\,\dd t.
\]
Since the support of $V_{\aa,\eps}$ is compact and one has $\dl(\eps) = \cO(\eps)$ as $\eps\arr 0$ by Theorem~\ref{thm}(i), it is not hard to see that there is a constant $C> 0$ such that the inequality
\[
	\big|e^{-|t-\dl(\eps)x_1'|} -e^{-|t|}\big| \le C\eps
\]
holds for all $x' = (x_1',x_2')\in \supp V_{\aa,\eps}$ and all $t\in\dR$. Hence we get
\[
\begin{aligned}
\|h_\eps\|^2 _{L^2(\dR^2) }&= \frac{v_1^2(d)(1+\cO(\eps))}{4\dl(\eps)^3}\int_\dR e^{-2|t|}\left|\int_{\dR^2}  V_{\aa,\eps}(x')g_\eps(x')\,\dd x'\right|^2\,\dd t\\
&=
\frac{v_1^2(d)(1+\cO(\eps))}{4\dl(\eps)^3}\left|\int_{\dR^2}  V_{\aa,\eps}(x')g_\eps(x')\,\dd x'\right|^2.
\end{aligned}
\]

\smallskip

\noindent{\it Step 3.}
In this step we analyze the terms $a_\eps, b_\eps$, and $c_\eps$ in~\eqref{eq:abc}. First we consider $a_\eps$; using the definitions of $\wh\sfR^I_{\aa,0}$ and of $\omg_{\aa,\eps}$ we obtain
\[
	a_\eps(x) = \frac{v_1(x_2)}{2\dl(\eps)}
	\int_{\dR^2} e^{-\dl(\eps)|x_1-x_1'|} v_1^2(x_2')V_{\aa,\eps}(x')\,\dd x'.
\]
Applying~\eqref{eq:norm_expansion} with $g_\eps(x) := v_1(x_2)\in L^2 _{\mathrm{loc}}(\dR^2)$ we get
\begin{equation}\label{eq:term_a}
\begin{aligned}
	\|a_\eps\|^2 _{L^2(\dR^2) } & = \frac{v_1^2(d)(1+\cO(\eps))}{4\dl(\eps)^3}\left[\int_{\dR^2} v_1(x_2)V_{\aa,\eps}(x)\,\dd x\right]^2\\
	&=\frac{\aa\eps^2v_1^4(d)(1+\cO(\eps))}{4\dl(\eps)^3}\left[\int_\dR f(x_1)\,\dd x_1\right]^2\\
	&=
	\frac{2(1+\cO(\eps))}{\eps\aa^2 v_1^2(d) }\left[\int_\dR f(x_1)\,\dd x_1\right]^{-1},
\end{aligned}
\end{equation}
where in the last step we used the expansion of $\dl(\eps)$ implicitly given in Theorem~\ref{thm}\,(i).

Next we consider the term $b_\eps$. In view of the decomposition $\sfN_\eps = \sfM_\eps + V_{\aa,\eps}\wh\sfR^{II}_{\aa,0}(\dl(\eps))|V_{\aa,\eps}|$ the subsequent analysis boils down to separate consideration of the terms
\[
b_\eps' := \wh\sfR^I_{\aa,0}(\dl(\eps))\sfM_\eps(\sfI-\sfN_\eps)^{-1}\omg_{\aa,\eps},\qquad b_\eps'':=\wh\sfR^I_{\aa,0}(\dl(\eps))V_{\aa,\eps}\wh\sfR^{II}_{\aa,0}(\dl(\eps))|V_{\aa,\eps}|(\sfI-\sfN_\eps)^{-1}\omg_{\aa,\eps}
\]
to which the quantity of interest splits, $b_\eps = b_\eps'+b_\eps''$. Applying~\eqref{eq:norm_expansion} to the first of these two terms, setting there $g_\eps = \aa^{-1} V_{\aa,\eps}\sfM_\eps (\sfI-\sfN_\eps)^{-1}\omg_{\aa,\eps} \in L^2(\dR^2) \subset L^2_{\mathrm{loc}} (\dR^2)$, we get
\[
\begin{aligned}
	&\|b_\eps'\|^2 _{L^2(\dR^2) }\! =\!
	\frac{v_1^2(d)(1+\cO(\eps))}{4\dl(\eps)^3}\times\\
	&\quad\times \left[
	\int_{\dR^2} V_{\aa,\eps}(x) v_1(x_2)\!
	\int_{\dR^2}\!m_{\dl(\eps)}(x_1,x_1') |V_{\aa,\eps}(x')|
	v_1(x_2')\big((\sfI\!-\!\sfN_\eps)^{-1}
	\omg_{\aa,\eps}\big)(x')\,\dd x'\,\dd x
	\right] ^2\!\!,
\end{aligned}
\]
where the function $m_{\dl(\eps)}$ is defined by~\eqref{eq:def_kernelM} with the bound determined by~\eqref{eq:def_Mprime}. Hence we get
\begin{equation}\label{eq:term2_bnd1}
	\|b_\eps'\|^2 _{L^2(\dR^2) } \le
	\frac{\aa v_1^8(d)M^2(1+\cO(\eps))}{4\dl(\eps)^3}\|(\sfI-\sfN_\eps)^{-1}\|^2
	\left(\int_{\dR^2} |V_{\aa,\eps}(x)|\,\dd x\right)^4
	\le C_1\frac{\eps^4}{\dl(\eps)^3} \le C_2\eps,
\end{equation}
with some constants $C_1,C_2 > 0$ independent of $\eps$; we used Theorem~\ref{thm}\,(i) in the last step.

Applying now~\eqref{eq:norm_expansion} to $b_\eps''$ with $g_\eps = \wh\sfR^{II}_{\aa,0}(\dl(\eps)) |V_{\aa,\eps}|(\sfI-\sfN_\eps)^{-1}\omg_{\aa,\eps}$ which belongs to $L^2 (\dR^2) \subset L^2_{\mathrm{loc} } (\dR) $ we get
\begin{equation}\label{eq:term3}
\|b_\eps''\|^2_{L^2(\dR^2) }
= \frac{v_1^2(d)(1+\cO(\eps))}{4\dl(\eps)^3}
\left[\int_{\dR^2} V_{\aa,\eps}(x)\big(\wh\sfR^{II}_{\aa,0}(\dl(\eps))|V_{\aa,\eps}|(\sfI-\sfN_\eps)^{-1}\omg_{\aa,\eps}\big)(x)\,\dd x  \right]^2\,.
\end{equation}
Introducing next the notation
\[
 \cA_\eps := \int_{\dR^2} V_{\aa,\eps}(x)\big(\wh\sfR^{II}_{\aa,0}(\dl(\eps))|V_{\aa,\eps}|(\sfI-\sfN_\eps)^{-1}\omg_{\aa,\eps}\big)(x)\,\dd x
\]
we infer with the help of the bound in Lemma~\ref{lem:R_norm_est}\,(i) that
\begin{equation}\label{eq:Aeps}
\begin{aligned}
	|\cA_\eps| &\le \big\|\wh\sfR^{II}_{\aa,0}(\dl(\eps))|V_{\aa,\eps}|(\sfI-\sfN_\eps)^{-1}\omg_{\aa,\eps} \big\|_{L^\infty (\dR^2)}\int_{\dR^2} |V_{\aa,\eps}(x)|\,\dd x\\
	&\le
C_3\big\||V_{\aa,\eps}|(\sfI-\sfN_\eps)^{-1}\omg_{\aa,\eps}\big\|_{L^2 (\dR^2)}
	  \int_{\dR^2} |V_{\aa,\eps}(x)|\,\dd x\\
	  &\le
	  C_3\sqrt{\aa}\|(\sfI-\sfN_\eps)^{-1}\|\cdot
	\| \omg_{\aa,\eps}\|_{L^2(\dR^2)}
	 \int_{\dR^2} |V_{\aa,\eps}(x)|\,\dd x
	   \le C_4 \eps^{3/2}
\end{aligned}
\end{equation}
holds with some constants $C_3,C_4 > 0$ independent of $\eps$. Substituting the estimate~\eqref{eq:Aeps} into~\eqref{eq:term3} we arrive at the bound
\begin{equation}\label{eq:term2_bnd2}
	\|b_\eps''\|^2 _{L^2(\dR^2)} \le C_5\frac{\eps^3}{\dl(\eps)^3} \le C_6
\end{equation}
with constants $C_5, C_6 > 0$ independent of $\eps > 0$; here again we used in the last step the asymptotics of $\dl(\eps)$ given in Theorem~\ref{thm}\,(i). Combining the bounds~\eqref{eq:term2_bnd1} and~\eqref{eq:term2_bnd2} we conclude that there exists a constant $C_7 > 0$ independent of $\eps$ such that
\begin{equation}\label{eq:term2_bnd}
	\|b_\eps\| _{L^2(\dR^2)}\le C_7,
\end{equation}
for all sufficiently small $\eps > 0$.

Finally, we consider the term $c_\eps$ the analysis of which is rather straightforward. As a consequence of the bound in Lemma~\ref{lem:R_norm_est}\,(i) we infer that there exist constants $C_7,C_8 > 0$ independent of $\eps$ such that
\begin{equation}\label{eq:term3_bnd}
	\|c_\eps\| _{L^2(\dR^2)} \le C_7\|(\sfI-\sfN_\eps)^{-1}\|\cdot\|\omg_{\aa,\eps}\|_{L^2(\dR^2)} \le C_8\sqrt{\eps}.
\end{equation}

\smallskip

\noindent{\it Step 4.}
In the last step we combine the expansion of $\psi_\eps$ in~\eqref{eq:f_expansion} obtained in Step 1 with the estimates of $a_\eps, b_\eps$ and $c_\eps$ obtained in Step 3 in order to get an asymptotic expansion of the eigenfunction $\sfH_{\aa,\eps}$ corresponding to its  unique simple eigenvalue in the limit $\eps \arr 0$. It follows from the expansion in~\eqref{eq:f_expansion} that an eigenfunction of $\sfH_{\aa,\eps}$ corresponding to this eigenvalue has the expansion
\begin{equation}\label{eq:psi_ef}
	\psi_\eps = a_\eps + b_\eps + c_\eps.
\end{equation}
Recall that the linear function $\wh\dl(\eps)$ is defined in the formulation of the theorem as
\[
	\wh\dl(\eps) =
	\eps\left(\frac{\aa v_1^2(d)}{2}\right)
	\int_{\dR} f(x_1)\,\dd x_1.
\]
We note that the asymptotics \eqref{eq:asymp_dl_eps} implies that $\dl(\eps) - \wh\dl(\eps) = \cO(\eps^2)$ holds as $\eps\arr0$. Let us introduce an auxiliary function,
\[
	\wh a_\eps := \wh\sfR^I_{\aa,0}(\wh\dl(\eps))\omg_{\aa,\eps} = \frac{v_1(x_2)}{2\wh\dl(\eps)}\int_{\dR^2} e^{-\wh\dl(\eps)|x_1-x_1'|} V_{\aa,\eps}(x')v_1^2(x_2')\,\dd x'
	\in L^2(\dR^2).
\]
Recall also that the function $a_\eps$ was defined by $a_\eps = \wh\sfR^I_{\aa,0}(\dl(\eps))\omg_{\aa,\eps}$. Next we show that $\wh a_\eps$ is close to $a_\eps$ in the needed sense. Using the resolvent identity we get
\begin{equation}\label{eq:wha_minus_a}
\begin{aligned}
	\wh a_\eps - a_\eps & = \big[
	\wh\sfR^I_{\aa,0}(\wh\dl(\eps))
	-\wh\sfR^I_{\aa,0}(\dl(\eps))
	\big]\omg_{\aa,\eps}\\
	& = \left[
	\sfR_{\aa,0}
	\left(\sqrt{-\mu_1+\wh\dl(\eps)^2}\right)
	-\sfR_{\aa,0}\left(
	\sqrt{-\mu_1+\dl(\eps)^2}\right)
	\right]\sfP_0\omg_{\aa,\eps}
	\\
	&=
	\left[\dl(\eps)^2-\wh\dl(\eps)^2\right]
	\sfR_{\aa,0}
	\left(\sqrt{-\mu_1+\wh\dl(\eps)^2}\right)
	\sfR_{\aa,0}\left(
	\sqrt{-\mu_1+\dl(\eps)^2}\right)
	\sfP_0\omg_{\aa,\eps} \\
	& =
	\left[\dl(\eps)^2-\wh\dl(\eps)^2\right]
	\sfR_{\aa,0}
	\left(\sqrt{-\mu_1+\wh\dl(\eps)^2}\right)
	a_\eps.
\end{aligned}
\end{equation}
Clearly, we have
\begin{equation}\label{eq:whaa_est1}
\dl(\eps)^2-\wh\dl(\eps)^2 = \cO(\eps^3),\qquad \eps \arr 0.
\end{equation}
Applying the spectral theorem and using the fact that $\mu_1$ is the lowest spectral point of $\sfH_{\aa,0}$ we obtain
\begin{equation}\label{eq:whaa_est2}
	\left\|
	\sfR_{\aa,0}\left(\sqrt{-\mu_1 +\wh\dl(\eps)^2}\right)
	\right\| = \frac{1}{\wh\dl(\eps)^2} = \cO(\eps^{-2}),\qquad \eps\arr 0.
\end{equation}
Combining~\eqref{eq:wha_minus_a} with~\eqref{eq:whaa_est1},~\eqref{eq:whaa_est2} and with~\eqref{eq:term_a} we infer that there exists a constant $C_9 > 0$ independent of $\eps > 0$ such that
\begin{equation}\label{eq:whaa_norm}
	\| \wh a_\eps - a_\eps \| \le C_9
\end{equation}
for all sufficiently small $\eps > 0$. As a consequence of~\eqref{eq:term_a} and~\eqref{eq:whaa_norm}, we conclude that
\begin{equation}\label{eq:wha_norm}
	\|\wh a_\eps\| _{L^2(\dR^2)} = \frac{\sqrt{2}}{\sqrt{\eps}\aa v_1(d)}\left[\int_\dR f(x_1)\,\dd x_1\right]^{-1/2} + \cO(1),\qquad\eps\arr 0.
\end{equation}
The functions $u_\eps$ and $v_\eps$, the leading term and the remainder, in the formulation of the theorem can be now represented as
\begin{equation} \label{uv_def}
	u_\eps = \frac{2\wh\dl(\eps)}{\sqrt{\eps}}\wh a_\eps,\qquad v_\eps =
	\frac{2\wh\dl(\eps)}{\sqrt{\eps}}
	\left(a_\eps-\wh a_\eps + b_\eps + c_\eps\right).
\end{equation}
In particular, we infer from~\eqref{eq:psi_ef} that $u_\eps + v_\eps$ is an eigenfunction of $\sfH_{\aa,\eps}$ for all sufficiently small $\eps > 0$ corresponding to the eigenvalue $\lm_1^\aa(\eps)$. It follows from~\eqref{eq:wha_norm} that
\[
	\|u_\eps\|_{L^2(\dR^2)} = \sqrt{2}v_1(d)\left[
	\int_{\dR} f(x_1)\,\dd x_1
	\right]^{1/2} + \cO(\sqrt{\eps}),\qquad \eps \arr 0.
\]
As a consequence of~\eqref{eq:term2_bnd},~\eqref{eq:term3_bnd} and~\eqref{eq:whaa_norm} we get that
\[
	\|v_\eps\|_{L^2(\dR^2)} = \cO(\sqrt{\eps}),\qquad \eps\arr 0;
\]
by that, the proof of the theorem is concluded.
\section{Proof of Theorem~\ref{thm:existence}}
\label{s:crit}

Let us pick a non-negative real-valued function $\chi\in C^\infty_0(\dR)$ such that $\chi(x) = 1$ for $x\in [-1,1]$ and $\supp\chi = [-2,2]$. Consider the following trial function	
\[
\psi_{\lm,\eps}(x_1,x_2) := \chi(\eps^3x_1)\big[1+\lm\eps f(x_1)\big]v_1(x_2),
\]
where the parameter $\lm > 0$ will be determined at a later stage. Under the regularity assumption $f\in W^{1,\infty}(\dR)$ it is easy to verify that $\psi_{\lm,\eps}\in H^1(\dR^2)$.
If for some $\lm > 0$ and all sufficiently small $\eps > 0$ the following inequality holds
\[
\cI_{\lm,\eps} :=
\int_{\dR^2}|\nb \psi_{\lm,\eps}|^2\,\dd x-\
\aa\int_{\Omg_\eps}|\psi_{\lm,\eps}|^2\,\dd x-\mu_1\int_{\dR^2}|\psi_{\lm,\eps}|^2\,\dd x < 0,
\]
then the operator $\sfH_{\aa,\eps}$ has by the min-max principle a discrete eigenvalue below $\mu_1$ for all sufficiently small $\eps > 0$. Moreover, in view of Proposition~\ref{prop:bs_bnd} the operator $\sfH_{\aa,\eps}$ has in this case a unique simple eigenvalue for all sufficiently small $\eps > 0$. Let $\eps > 0$ be so small that $\supp f\subset[-1/\eps^3,1/\eps^3]$. Recall that $v_1\in H^2(\dR)$ is the normalized ground-state eigenfunction of the one-dimensional Schr\"odinger operator $\sfh_\aa$. It follows from the embedding of $H^2(\dR)$ into $C^1(\dR)$ that $v_1\in C^1(\dR)$. Moreover, the eigenvalue equation $-v_1'' -\aa\chi_{[0,d]}v_1 = \mu_1 v_1$ implies that the second derivative of $v_1$ is continuous on the intervals $(-\infty,0]$, $[0,d]$ and $[d,\infty)$.

We substitute the expression for $\psi_{\lm,\eps}$ into the formula for $\cI_{\lm,\eps}$,
\[
\begin{aligned}
\cI_{\lm,\eps} & =
\eps^6\int_\dR|\chi'(\eps^3 x_1)|^2\,\dd x_1 + \int_{\dR}\lm^2\eps^2|f'(x_1)|^2\,\dd x_1\\
&\qquad + \int_{\dR}\int_{\dR}
|\chi(\eps^3x_1)|^2|1+\lm\eps f(x_1)|^2|v_1'(x_2)|^2\,\dd x_1\,\dd x_2\\
&\qquad\quad-\aa\int_{\dR}
\int_0^d|\chi(\eps^3 x_1)|^2|1+\lm\eps f(x_1)|^2|v_1(x_2)|^2\,\dd x_2\,\dd x_1 \\
&\qquad\qquad-\aa\int_{\dR} \int_{d}^{d+\eps f(x_1)}
|1+\lm\eps f(x_1)|^2|v_1(x_2)|^2\,\dd x_2\,\dd x_1\\
&\qquad\qquad\quad-\mu_1\int_{\dR}\int_{\dR}|\chi(\eps^3x_1)|^2|1+\lm\eps f(x_1)|^2|v_1(x_2)|^2\,\dd x_1\,\dd x_2\\
&=
\eps^3\int_\dR|\chi'(x_1)|^2\,\dd x_1 + \int_{\dR}\lm^2\eps^2|f'(x_1)|^2\,\dd x_1
 -\aa\int_{\dR} \int_{d}^{d+\eps f(x_1)}
|1+\lm\eps f(x_1)|^2|v_1(x_2)|^2\,\dd x_2\,\dd x_1\\
&
=\eps^3\int_\dR|\chi'(x_1)|^2\,\dd x_1 + \int_{\dR}\lm^2\eps^2|f'(x_1)|^2\,\dd x_1\\
&\qquad -\aa\int_{\dR}
|1+\lm\eps f(x_1)|^2
\big[\eps  f(x_1)|v_1(d)|^2 + v_1(d)v_1'(d)\eps^2 f^2(x_1)+\cO_{\rm u}(\eps^3) \big]\,\dd x_1\,,
\end{aligned}
\]
where we used in the last step the Taylor expansion of $v_1$ up to the second term with a remainder in the neighbourhood of the point $x_2=d$; here $\cO_{\rm u}(\eps^3)$ means a compactly supported function of $x_1$ which can uniformly bounded by a multiple of $\eps^3$.

In this way, we derive the expansion
\[
\cI_{\lm,\eps} = \eps^2\left[
\lm^2\int_{\dR}|f'(x_1)|^2\,\dd x_1-\aa\big(
2\lm|v_1(d)|^2+ v_1(d)v_1'(d)\big)\int_{\dR}|f(x_1)|^2\,\dd x_1\right] +\cO(\eps^3).
\]
It is clear from the eigenvalue equation that $v_1(x) = Ce^{-\sqrt{-\mu_1}x}$ for all $x > d$ and some constant $C > 0$. Hence we get that
\[
v_1'(d) = -\sqrt{-\mu_1}v_1(d),
\]
and we finally end up with the expansion
\[
\cI_{\lm,\eps} = \eps^2\left[
\lm^2\int_{\dR}|f'(x_1)|^2\,\dd x_1-\aa|v_1(d)|^2
\big(
2\lm-\sqrt{-\mu_1}\big)\int_{\dR}|f(x_1)|^2\,\dd x_1\right] +\cO(\eps^3).
\]
The quantity $\cI_{\lm,\eps}$ is negative for all sufficiently small $\eps > 0$ provided that
\[	
\frac{\int_{\dR}|f'(x_1)|^2\,\dd x_1}{\int_{\dR}|f(x_1)|^2\,\dd x_1} <
\aa|v_1(d)|^2\left(\frac{2}{\lm}-\frac{\sqrt{-\mu_1}}{\lm^2}\right).
\]
Maximizing the right hand side with respect to $\lm$ we find that the maximum is positive and is achieved for $\lm = \sqrt{-\mu_1}$. The final sufficient condition that we get is
\[	
\frac{\int_{\dR}|f'(x_1)|^2\,\dd x_1}{\int_{\dR}|f(x_1)|^2\,\dd x_1} <
	\frac{\aa|v_1(d)|^2}{\sqrt{-\mu_1}},
\]
by which the theorem is proved.

\subsection*{Acknowledgement}
The research of P.\,E. and V.\,L. was supported by the Czech Science Foundation (GA\v{C}R) within the project 21-07129S; the former is also obliged to the EU project  CZ.02.1.01/0.0/0.0/16\textunderscore 019/0000778.
S.\,K. acknowledges the financial support from the program of the Polish Ministry
of Science and Higher Education under the name Regional Initiative of Excellence   in 2019-2022, Project No. 03/RID/2018/19.


\newcommand{\etalchar}[1]{$^{#1}$}

 \end{document}

%% file: commands.tex
\newcommand\nb{\nabla}
\newcommand{\beq}{\begin{equation} \begin{split}}
\newcommand{\eeq}{\end{split} \end{equation}}

\newcommand\Omg{\Omega}

\makeatletter
\def\section{\@startsection{section}{1}\z@{.9\linespacing\@plus\linespacing}%
	{.7\linespacing} {\fontsize{13}{14}\selectfont\bfseries\centering}}
\def\paragraph{\@startsection{paragraph}{4}%
	\z@{0.3em}{-.5em}%
	{$\bullet$ \ \normalfont\itshape}}

\@addtoreset{equation}{section}
\makeatother

\renewcommand\and{\qquad\text{and}\qquad}

\newcommand\sm{\setminus}

\newcommand\dl{\delta}

\newcommand{\comm}[1]{}

\def\sfH{\mathsf{H}}

\def\bm1{\mathbbm{1}}
\def\G{\Gamma}

\def\s{\sigma}
\def\sfh{\mathsf{h}}

\def\omg{\omega}

\def\arr{\rightarrow}

\renewcommand{\gg}{{\gamma}}

\def\aa{\alpha}
\def\lm{\lambda}

\def\s{\sigma}

\def\sess{\sigma_{\rm ess}}

\def\kp{\kappa}

\def\sfH{\mathsf{H}}
\def\sfh{\mathsf{h}}

\def\sfP{\mathsf{P}}
\def\dd{{\mathsf{d}}}

\def\omg{\omega}

\def\sfP{\mathsf{P}}

\newcounter{counter_a}
\newenvironment{myenum}{\begin{list}{{\rm(\roman{counter_a})}}%
{\usecounter{counter_a}
\setlength{\itemsep}{1.ex}\setlength{\topsep}{0.8ex}
\setlength{\leftmargin}{5ex}\setlength{\labelwidth}{5ex}}}{\end{list}}

\usepackage[latin1]{inputenc}
\usepackage[T1]{fontenc}

\newcommand{\eg}{{\it e.g.}\,}

\newcommand{\cf}{{\it cf.}\,}

\numberwithin{figure}{section}
\numberwithin{equation}{section}
\theoremstyle{plain}
\newtheorem*{thm*}{Theorem}
\newtheorem{thm}{Theorem}[section]

\theoremstyle{remark}
\newtheorem{remark}[thm]{Remark}
\theoremstyle{plain}

\newcommand{\supp}{\mathrm{supp}\,}
\newcommand{\beu}{\begin{equation*}}
\newcommand{\eeu}{\end{equation*}}
\newcommand{\besu}{\begin{equation*}
\begin{aligned}}
\newcommand{\eesu}{\end{aligned}
\end{equation*}}
\newcommand{\bes}{\begin{equation}
\begin{aligned}}
\newcommand{\ees}{\end{aligned}
\end{equation}}

\newcommand\cA{\mathcal A}
\newcommand\cB{\mathcal B}

\newcommand\cG{\mathcal G}

\newcommand\cJ{\mathcal J}
\newcommand\cL{\mathcal L}
\newcommand\cM{\mathcal M}

\newcommand\frh{\mathfrak h}

\newcommand\eps{\varepsilon}

\newcommand\ov{\overline}

\newcommand\wh{\widehat}

\newcommand\sign{{\rm sign\,}}

\newcommand\void[1]{}

\def\ov{\overline}
\def\eps{\varepsilon}
\def\ran{{\rm ran\,}}

      \def\dC{{\mathbb C}}

   \def\dN{{\mathbb N}}   
      \def\dR{{\mathbb R}}

   \def\sfH{{\mathsf H}}   \def\sfI{{\mathsf I}}
      \def\sfL{{\mathsf L}}
\def\sfM{{\mathsf M}}   \def\sfN{{\mathsf N}}   
\def\sfP{{\mathsf P}}      \def\sfR{{\mathsf R}}

\def\cA{{\mathcal A}}   \def\cB{{\mathcal B}}   
      
\def\cG{{\mathcal G}}      \def\cI{{\mathcal I}}
\def\cJ{{\mathcal J}}      \def\cL{{\mathcal L}}
\def\cM{{\mathcal M}}      \def\cO{{\mathcal O}}

\newcommand{\dom}{\mathrm{dom}\,}